\documentclass[final,hidelinks]{siamltex}
\usepackage{amsfonts}
\usepackage{amsmath}
\usepackage{amssymb}
\usepackage{color, colortbl}
\usepackage{graphicx}
\graphicspath{{figures/}}
\usepackage{hyperref}
\usepackage{caption}
\usepackage{subcaption}

\title{On the application of Laguerre's method to the polynomial eigenvalue problem}
\author{Thomas R. Cameron\thanks{The College of Idaho} \and Nikolas Steckley\thanks{Steckley \& Associates}}

\newtheorem{ex}{\emph{Example}}[section]
\DeclareMathOperator{\rP}{revP}
\DeclareMathOperator{\ra}{rev\alpha}
\begin{document}
\maketitle

\begin{abstract}
The polynomial eigenvalue problem arises in many applications and has received a great deal of attention over the last decade. The use of root-finding methods to solve the polynomial eigenvalue problem dates back to the work of Kublanovskaya (1969, 1970) and has received a resurgence due to the work of Bini and Noferini (2013). In this paper, we present a method which uses Laguerre iteration for computing the eigenvalues of a matrix polynomial. An effective method based on the numerical range is presented for computing initial estimates to the eigenvalues of a matrix polynomial. A detailed explanation of the stopping criteria is given, and it is shown that under suitable conditions we can guarantee the backward stability of the eigenvalues computed by our method. Then, robust methods are provided for computing both the right and left eigenvectors and the condition number of each eigenpair. Applications for Hessenberg and tridiagonal matrix polynomials are given and we show that both structures benefit from substantial computational savings. Finally, we present several numerical experiments to verify the accuracy of our method and its competitiveness for solving the roots of a polynomial and the tridiagonal eigenvalue problem.   
\end{abstract}

\pagestyle{myheadings}
\thispagestyle{plain}
\markboth{THOMAS R. CAMERON AND NIKOLAS STECKLEY}{THE POLYNOMIAL EIGENVALUE PROBLEM}

\section{Introduction}
The \emph{polynomial eigenvalue problem} consists of computing the eigenvalues, and often eigenvectors, of an $n\times n$ matrix polynomial of degree $d$:
\begin{equation}\label{eq:matpoly}
P(\lambda)=\underset{i=0}{\overset{d}\sum}\lambda^{i}A_{i}\text{, where }~A_{i}\in\mathbb{C}^{n\times n}~\text{ and }~A_{d}\neq 0.
\end{equation}
An eigenvalue of $P(\lambda)$ is any scalar $\lambda\in\mathbb{C}$ such that $\det P(\lambda)=0$. Any nonzero vector $x\in\ker P(\lambda)$ is an eigenvector corresponding to $\lambda$. The \emph{algebraic multiplicity} of $\lambda$ is its multiplicity as a root of $\det P(\lambda)$, and the \emph{geometric multiplicity} of $\lambda$ is the dimension of $\ker P(\lambda)$. 

Throughout this paper we assume that the matrix polynomial is \emph{regular}, that is, $\det P(\lambda)$ is not the constant zero polynomial, and therefore the set of all eigenvalues is a subset of the extended complex plane with cardinality $nd$. Infinite eigenvalues of~\eqref{eq:matpoly} can occur if the leading coefficient matrix is singular and are defined as the zero eigenvalues of the reversal polynomial,
\begin{equation}\label{eq:revpoly}
\rP(\rho)=\underset{i=0}{\overset{d}\sum}\rho^{d-i}A_{i}.
\end{equation}

Computing an eigenpair $(\lambda,x)$ is useful for a large range of applications~\cite{Betcke2013}. Of extreme importance are special cases of the polynomial eigenvalue problem, such as finding the roots of a scalar polynomial ($n=1$) and solving the linear eigenvalue problem ($d=1$). What's more, the established techniques for solving these special case problems motivate two current approaches for solving the polynomial eigenvalue problem: linearization and root finding methods. 

The linearization of a matrix polynomial results in an equivalent linear eigenvalue problem which is often solved using QZ iteration. Algorithms which adopt this approach have computational complexity $O(d^{3}n^{3})$ and include the popular MATLAB functions QUADEIG~\cite{Hammarling2013} and POLYEIG~\cite{Dedieu2003,Meerbergen2001,Tisseur2000}. More recently, it was shown that exploiting the inherent structure in the companion linearization results in a $O(d^{2}n^{3})$ algorithm~\cite{Aurentz2016}. However, often the original matrix polynomial comes with structure worth exploiting and, in general, the companion linearization does not preserve this structure. Furthermore, the conditioning of the larger linear problem can be worse than the original problem~\cite{Mackey2015}.

To our knowledge, Kublanovskaya was the first to use root-finding methods to solve the polynomial eigenvalue problem~\cite{Kublanovskaya1970} when she suggested the use of the QR decomposition with column pivoting and Newton’s method to compute an eigenvalue of the matrix polynomial. Improvements to this method were given, and quadratic convergence was shown, by Jain, Singhal, and Huseyin~\cite{Jain1983}. More recently, a cubic convergent algorithm using the Ehrlich-Aberth method to compute the eigenvalues of a matrix polynomial was presented~\cite{Bini2013-1}. These root-finding methods are rather inefficient, though, as they compute one eigenvalue at a time, each requiring several $O(n^{3})$ factorizations. However, certain structures in the original problem can be exploited, thus increasing the efficiency of these methods. In addition to preserving the structure, root-finding methods have the advantage of preserving the size and conditioning of the original problem

Root-finding methods exhibit a high level of accuracy, thus making them useful in the context of iterative refinement of computed eigenvalues and eigenvectors. They have been shown to be cost efficient for solving large degree polynomial eigenvalue problems~\cite{Bini2013-1}, and are the driving force behind what is perhaps the fastest and most accurate algorithm for solving the nonsymmetric tridiagonal eigenvalue problem~\cite{Bini2005}. Furthermore, both Laguerre's method and the Ehrlich-Aberth method have been used as an accurate and efficient method for solving the quadratic tridiagonal eigenvalue problem~\cite{Plestenjak2006}. 

In this paper, we propose a root-finding algorithm which uses Laguerre iteration to solve the polynomial eigenvalue problem. Our method is motivated by the previous work of Bini and Noferini~\cite{Bini2013-1, Bini2013-2}, Gary~\cite{Gary1965}, and Parlett~\cite{Parlett1964}. In \S\ref{sec:LMPEP} we present a method for computing the Laguerre iterate of an approximate eigenvalue. We provide robust methods for computing the corresponding right and left eigenvectors, backward error, and condition estimates. Both Hessenberg and tridiagonal structures are considered, and it is shown that Hyman's method can be used to obtain significant computational savings. In \S\ref{subsec:IE} we develop a method based on the numerical range for computing initial estimates to the eigenvalues of a matrix polynomial. Under suitable conditions these initial estimates are no bigger in absolute value than the upper Pellet bounds. Finally, an a priori check for both zero and infinite eigenvalues is implemented and comparisons are made to the approach developed in~\cite{Bini2013-1,Bini2013-2}.

In \S\ref{sec:stability} we discuss the stability of our method.
Specifically, we show that our method is robust against overflow and that under suitable conditions we can guarantee the backward stability of the eigenvalues computed. In \S\ref{sec:NE} numerical experiments are provided to verify the accuracy and cost analysis of our method. Additionally, comparisons are made to the methods in~\cite{Aurentz2015,Bini1996} and~\cite{Bini2005,Plestenjak2006} to verify the effectiveness of our method for computing the roots of a polynomial and solving the tridiagonal polynomial eigenvalue problem, respectively.

\section{Laguerre's method applied to the polynomial eigenvalue problem}\label{sec:LMPEP} 
Laguerre's method has a rich history originating with the work of Edmond Laguerre~\cite{Laguerre1898}. Laguerre's method has incredible virtues including guaranteed global convergence when all roots are real~\cite{Acton1970}, and when these are simple zeros this method is known to exhibit local cubic convergence. In practice, the complex iterations seem as powerful as the real one's~\cite{Parlett1964}. Both Numerical Recipes (zroots) and the NAG 77 Library (C02AFF) employ a modified Laguerre method to compute the roots of a scalar polynomial. In 1964-65, Laguerre's method was applied to the linear eigenvalue problem, both in the monic~\cite{Parlett1964} and non-monic~\cite{Gary1965} cases. Now we apply Laguerre's method to the polynomial eigenvalue problem.

Since $P(\lambda)$ is assumed to be regular, the polynomial $p(\lambda)=\det P(\lambda)$ has at most $N_{1}\leq nd$ roots, where $N_{1}+N_{2}=nd$ and $N_{2}$ is the number of infinite eigenvalues. Given an approximation $\lambda$ to one of the roots of $p(\lambda)$, Laguerre's method uses $p(\lambda)$, $p^{'}(\lambda)$, and $p^{''}(\lambda)$ to obtain a better approximation. Following the development in~\cite{Parlett1964}, we define the following:
\begin{equation}\label{eq:S1}
S_{1}(\lambda)=\frac{p^{'}(\lambda)}{p(\lambda)}=\underset{i=1}{\overset{N_{1}}\sum}\frac{1}{\lambda-r_{i}},
\end{equation}
where $r_{1},\ldots,r_{N_{1}}$ are the roots of $p(\lambda)$, and
\begin{equation}\label{eq:S2}
S_{2}(\lambda)=-\left(\frac{p^{'}(\lambda)}{p(\lambda)}\right)^{'}=\underset{i=1}{\overset{N_{1}}\sum}\frac{1}{(\lambda-r_{i})^{2}}.
\end{equation}
Then the next approximation is given by
\begin{equation}\label{eq:lit}
\hat{\lambda}=\lambda-\frac{N_{1}}{S_{1}\pm\sqrt{(N_{1}-1)(N_{1}S_{2}-S_{1}^{2})}},
\end{equation}
where the sign of the square root is chosen to maximize the magnitude of the denominator. We call $\hat{\lambda}$ the \emph{Laguerre iterate} of $\lambda$. Once the roots $r_{1},\ldots,r_{k}$ have been found, we deflate the problem by subtracting
\[
\underset{i=1}{\overset{k}\sum}\frac{1}{\lambda-r_{i}}~\text{ and }~\underset{i=1}{\overset{k}\sum}\frac{1}{(\lambda-r_{i})^{2}}
\]
from equations~\eqref{eq:S1} and~\eqref{eq:S2}, respectively.

The undesirable numerical properties of the determinant are well-known, and it is for these reasons that we do not work with the polynomial $p(\lambda)$ directly. Rather, an effective method for computing equations~\eqref{eq:S1}--\eqref{eq:S2} can be derived from Jacobi's formula:
\begin{equation}\label{eq:clit}
\begin{split}
\frac{p^{'}(\lambda)}{p(\lambda)}&=\text{trace}\left(X_{1}(\lambda)\right),\\
-\left(\frac{p^{'}(\lambda)}{p(\lambda)}\right)^{'}&=\text{trace}\left(X_{1}^{2}(\lambda)-X_{2}(\lambda)\right),
\end{split}
\end{equation}
where $P(\lambda)X_{1}(\lambda)=P^{'}(\lambda)$ and $P(\lambda)X_{2}(\lambda)=P^{''}(\lambda)$. The first formula in~\eqref{eq:clit} can be found in~\cite{Bini2013-1}. The second formula follows from the first by using the derivative product rule and noting that $\left(P^{-1}(\lambda)\right)^{'}P^{'}(\lambda)=-X_{1}^{2}(\lambda)$. Note that only the diagonal entries of $X_{1}(\lambda)^{2}$ are needed in~\eqref{eq:clit}, which is significantly less expensive than computing the matrix product.

In general, the method we propose begins with initial estimates to the eigenvalues of the matrix polynomial. Then, proceeding one at a time, the Laguerre iteration of each eigenvalue approximation is computed, which requires solving the matrix equations in~\eqref{eq:clit}. Each eigenvalue is updated until at least one of the stopping criteria are met (see \S\ref{sec:ESC}). Locally, if the root is simple, convergence is cubic; otherwise, it is linear. Furthermore, in practice, the total number of iterations needed to compute all eigenvalues is proportional to the product $nd$; therefore our method has computational complexity $O(dn^{4}+d^{2}n^{3})$.

In \S\ref{sec:HessTri}, we show that significant computational savings can be obtained from Hyman's method for both Hessenberg and tridiagonal matrix polynomials. In addition, the general method can easily be specialized for scalar polynomials, and we are left with a method that has computational complexity $O(d^{2})$. 

\subsection{Eigenvectors, Stopping Criteria, and Condition Numbers}\label{sec:ESC}
Denote by $\lambda\in\mathbb{C}$ an approximate eigenvalue, and define
\begin{equation}\label{eq:qrf}
QR=P(\lambda)E,
\end{equation}
where $E$ is a permutation matrix such that $|r_{11}|\geq\cdots\geq|r_{nn}|$. If $|r_{nn}|<\tau$, where $\tau$ is some predetermined tolerance, then we say that the approximate eigenvalue has converged. This constitutes our first stopping criterion. In \S\ref{sec:bs} we define $\tau$ and show that the first stopping criterion guarantees that the backward error in the approximate eigenpair is very small. 

Given that $\lambda$ has converged, we compute the corresponding right and left eigenvectors by
\begin{equation}\label{eq:ev1}
x=E\hat{x}~\text{ and }~y=Qe_{n},
\end{equation}
where
\[
R(1:n-1,1:n-1)\hat{x}(1:n-1)=-R(1:n-1,n)
\]
$\hat{x}(n)=1$ and $e_{n}$ is the $nth$ standard basis vector.

This approach works well when $|r_{nn}|<\tau$, and while this is sufficient to guarantee that the approximate eigenvalue has converged, it is not necessary. Indeed, there exist upper triangular matrices that are ``nearly'' rank deficient, yet none of the main diagonal entries are extremely small~\cite{Watkins2010}. For this reason, we introduce a second stopping criterion based on an upper bound estimation of the backward error in the eigenvalue approximation.

If an approximate eigenvector has not been computed, then it follows from~\cite{Tisseur2000}[Lemma 3] that for any nonzero vector $b\in\mathbb{C}^{n}$ the backward error in the eigenvalue approximation is bounded above by
\begin{equation}\label{eq:berrUB}
\frac{\left\Vert b \right\Vert_{2}}{\alpha\left\Vert P(\lambda)^{-1}b \right\Vert_{2}},
\end{equation}
where $\alpha=\sum_{i=0}^{d}|\lambda|^{i}\left\Vert A_{i} \right\Vert_{2}$.

Suppose $\lambda$ is an approximate eigenvalue such that the upper bound on its backward error, and therefore its backward error, is less than double precision \emph{unit roundoff}: $\epsilon=2^{-53}$; then we say that $\lambda$ has converged. This constitutes our second stopping criterion. In practice we take the min of~\eqref{eq:berrUB} over three nonzero vectors $b\in\mathbb{C}^{n}$.

If none of the diagonal entries in the matrix $R$ are less than $\tau$, then~\eqref{eq:ev1} is not suitable for computing the corresponding eigenvectors. Rather, we compute the singular vectors corresponding to the smallest singular value of $P(\lambda)$. With the QR factorization with column pivoting in~\eqref{eq:qrf}, we apply inverse iteration to
\begin{equation}\label{eq:inviter}
E(R^{*}R)E^{T}~\text{ and }~Q(RR^{*})Q^{*},
\end{equation}
to compute the right and left singular vectors, respectively. Our experience indicates that using~\eqref{eq:ev1} to form initial estimates for the inverse iterations results in quick convergence to excellent eigenvector approximations.

Now, suppose that $\lambda$ is an approximate eigenvalue which satisfies
\begin{equation}\label{eq:ssc}
|\hat{\lambda}-\lambda|<\epsilon|\lambda|
\end{equation}
where $\hat{\lambda}$ is the Laguerre iterate defined in~\eqref{eq:lit}. Then no significant change to the current eigenvalue approximation is made, and we say that $\lambda$ has converged. This constitutes our third stopping criterion. In this case, or in the case where some predefined maximum number of iterations has been reached, we cannot make a strong statement about the approximations backward error. At this point, our best option is to proceed by computing the singular vectors corresponding to the smallest singular value of $P(\lambda)$ using the inverse iteration described in~\eqref{eq:inviter}.  

In summary, given an approximate eigenvalue, we compute the QR factorization with column pivoting in~\eqref{eq:qrf}. If any of the three stopping criteria are met, or the maximum number of iterations is reached, then we cease to update the eigenvalue approximation and compute corresponding right and left eigenvectors. Otherwise, we use the QR factorization to update the eigenvalue approximation by solving~\eqref{eq:clit} and computing the Laguerre iterate. 

Several remarks are in order. First, the norm of the matrix coefficients are only computed once, and in practice, we replace the matrix 2-norm with the Frobenius norm. Second, the definition we've given for $\alpha$ in~\eqref{eq:berrUB} results in a relative normwise measurement of the backward error (see \S\ref{sec:bs}). Finally, we note that the addition of computing the eigenvectors for each approximate eigenvalue has not changed the computational complexity of our method, which is $O(dn^{4}+d^{2}n^{3})$.

Once the approximate eigenvalue has converged and the corresponding right and left eigenvectors have been computed, we report each eigenvalue's condition number. It follows from~\cite{Tisseur2000}[Theorem 5], that the normwise condition number of a nonzero finite simple eigenvalue is given by
\begin{equation}\label{eq:con}
\kappa(\lambda,P)=\frac{\alpha\left\Vert x \right\Vert_{2}\left\Vert y \right\Vert_{2}}{|\lambda||y^{*}P^{'}(\lambda)x|}.
\end{equation}
For simple zero and infinite eigenvalues, we report
\[
\frac{\left\Vert x \right\Vert_{2}\left\Vert y \right\Vert_{2}}{|y^{*}x|}
\]
as the condition number, where $x$ and $y$ are right and left eigenvectors corresponding to the zero eigenvalues of the matrices $A_{0}$ and $A_{d}$, respectively.

\subsection{Initial Estimates}\label{subsec:IE}
A root-finding method's performance is greatly influenced by its initial estimates. In~\cite{Bini2013-1,Bini2013-2}, it is suggested to use the Newton polygon of a polynomial formed from the norm of the coefficient matrices to obtain initial estimates to the eigenvalues of the matrix polynomial $P(\lambda)$. In this section, we review this Newton polygon approach, since we will use it to form initial estimates in the scalar case. However, for matrix polynomials we propose a new method, motivated by the numerical range of the matrix polynomial, for computing the initial estimates. 

\subsubsection{Newton Polygon}\label{subsubsec:NP}
The Newton polygon approach works by placing initial estimates on circles of suitable radii. We quantify what constitutes suitable radii from the Pellet bounds for matrix polynomials.  
\begin{theorem}\label{thm:Pellet}
Let $P(\lambda)$ be an $n\times n$ matrix polynomial of degree $d\geq 2$, where $A_{0}\neq 0$. For each $k\in\left\{0,1,\ldots,d\right\}$ such that $A_{k}$ is nonsingular, consider the equation
\begin{equation}\label{eq:Pellet}
\left\Vert A_{k}^{-1}\right\Vert^{-1}\mu^{k}=\underset{i\neq k}{\sum}\left\Vert A_{i}\right\Vert \mu^{i},
\end{equation}
where $\left\Vert \cdot\right\Vert$ is any induced matrix norm.
\begin{romannum}
\item If $k=0$ there exists one real positive solution $r$, and $P(\lambda)$ has no eigenvalues of moduli less than $r$.
\item If $0<k<d$ there are either no real positive solutions or two real positive solutions $r_{1}\leq r_{2}$. In the latter case, $P(\lambda)$ has no eigenvalues in the annulus $\mathcal{A}(r_{1},r_{2})=\left\{z\in\mathbb{C} : r_{1}<|z|<r_{2}\right\}$.
\item If $k=d$, then there exists one real positive solution $R$, and $P(\lambda)$ has no eigenvalues of moduli greater than $R$. 
\end{romannum}
\end{theorem}
\noindent
A proof of Theorem~\ref{thm:Pellet} can be found in~\cite{Cameron2015, Melman2013}. Moreover, it was noted in~\cite{Bini2013-2} that the bounds in Theorem~\ref{thm:Pellet} can be sharpened if~\eqref{eq:Pellet} is replaced by 
\begin{equation}\label{eq:Pellet2}
\mu^{k}=\underset{i\neq k}{\sum}\left\Vert A_{k}^{-1}A_{i}\right\Vert \mu^{i}.
\end{equation}

Let $k_{0},\ldots,k_{q}$ be values of $k$ such that $A_{k}$ is nonsingular and there exists real positive solution(s) $s_{k_{i}}\leq t_{k_{i}}$ to~\eqref{eq:Pellet2}. Then $t_{k_{i-1}}\leq s_{k_{i}}$ for $i=1,\ldots,q$, and there are $n(k_{i}-k_{i-1})$ eigenvalues of $P(\lambda)$ in the closure of the annulus $\mathcal{A}(t_{k_{i}-1},s_{k_{i}})$. If for $k=0$ or $k=d$ the matrix $A_{k}$ is singular, then $t_{0}=0$ or $s_{d}=\infty$, respectively. Computing the value of $s_{k}$ and $t_{k}$ is expensive, since it requires solving several matrix and polynomial equations. However, a cheap algorithm for approximating $s_{k}$ and $t_{k}$ was proposed in~\cite{Melman2014}. 

For the scalar case ($n=1$) there is an alternative to computing the values of $s_{k}$ and $t_{k}$. Consider the polynomial $w(\lambda)=\underset{i=0}{\overset{d}\sum}a_{i}\lambda^{i}$, where $a_{0}a_{d}\neq 0$. The \emph{Newton polygon} associated with this polynomial is the upper convex hull of the discrete set $\left\{(i,\log|a_{i}|):i=0,1,\ldots,d\right\}$. Let $0=k_{0}<k_{1}<\cdots<k_{q}=d$ denote the abscissas of the vertices of the Newton polygon, and define the radii
\begin{equation}\label{eq:npradii}
r_{i}=\left|\frac{a_{k_{i-1}}}{a_{k_{i}}}\right|^{\frac{1}{k_{i}-k_{i-1}}},
\end{equation}
for $i=1,\ldots,q$. Then $(k_{i}-k_{i-1})$ initial estimates to the roots of $w(\lambda)$ are placed on circles centered at $0$ with radius $r_{i}$. In~\cite[Theorem 1.2]{Bini2013-2} they show that these estimates lie within the Pellet bounds for the polynomial $w(\lambda)$, and in~\cite{Bini1996} they establish the efficiency of these initial estimates for solving the roots of a polynomial. 

In~\cite{Bini2013-1,Bini2013-2} they generalize this approach for a specific class of matrix polynomials, and in~\cite{Noferini2015} to general matrix polynomials. In practice, the idea is simple. Let 
\[
w(\lambda)=\underset{i=0}{\overset{d}\sum}\left\Vert A_{i}\right\Vert \lambda^{i}.
\] 
Then, $n(k_{i}-k_{i-1})$ initial estimates to the eigenvalues of $P(\lambda)$ are placed on circles centered at zero with radius $r_{i}$, for $i=1,\ldots,q$, where both $k_{i}$ and $r_{i}$ are defined as in~\eqref{eq:npradii} with reference to the Newton polygon associated with the polynomial $w(\lambda)$.

\subsubsection{Numerical Range}\label{subsubsec:NR}
The \emph{numerical range} of a matrix polynomial is the set
\begin{equation}\label{eq:NR}
W(P)=\{\lambda\in\mathbb{C}\colon x^{*}P(\lambda)x=0\text{, for some nonzero vector } x\in\mathbb{C}^{n}\}
\end{equation}
which clearly contains the set of all eigenvalues. Under suitable conditions, see Theorem~\ref{thm:SAMPIE}, the roots of the quadratic form $x^{*}P(\lambda)x$, where $x\in\mathbb{C}^{n}$ is of unit length, are no bigger in absolute value than the upper Pellet bound, see Theorem~\ref{thm:Pellet}. In practice, we make use of the columns of $Q=[q_{j}]_{j=1}^{n}$, already obtained from the QR factorization of the constant and leading coefficient matrices, see \S\ref{sec:zeroinf}. Initial estimates to the finite eigenvalues are computed as the roots of $q_{j}^{*}P(\lambda)q_{j}$ for $j=1,\ldots,n$.

If $P(z)=zI-A$, then $W(P)$ coincides with the classical numerical range (field of values) of the matrix $A$, which has wonderful properties including convexity and connectedness. In general, however, the numerical range of a matrix polynomial need not have these properties and is bounded if and only if the field of values of the leading coefficient matrix does not contain the origin. For a detailed introduction to the numerical range of a matrix polynomial and its geometric properties see~\cite{Li1994}. 

It is highly nontrivial to give a complete description of the set $W(P)$. Despite this, we have experienced great success using elements from the numerical range as initial estimates for the eigenvalues we wish to compute. This seems to be a consequence of the habitual nature of elements from the numerical range to adhere to the geometric structure of the spectrum. To exemplify this statement, consider the hyperbolic matrix polynomial $P(\lambda)$, which by definition has a numerical range that satisfies $W(P)\subset\mathbb{R}$. Then, it is clearly advantageous to use initial estimates from the numerical range over elements on a circle in the complex plane.

Even more revealing, the numerical range of a hyperbolic matrix polynomial is split into $d$ ``spectral regions'' each containing a root of $x^{*}P(\lambda)x$. Each spectral region is an interval (possibly degenerate) on the real line that contains $n$ eigenvalues of $P(\lambda)$~\cite{Li1994}. In general, singling out a part of $W(P)$ containing precisely $k$ roots of $x^{*}P(\lambda)x$ for any unit vector $x\in\mathbb{C}^{n}$ and separated from the rest of $W(P)$ by a circle establishes the existence of a spectral divisor of order $k$ whose spectrum lies in that region~\cite{Markus1988}[\S~26.4]. For simplicity, we also reference this region as a spectral region.

In what follows, we provide three example problems from the NLEVP package~\cite{Betcke2013} to illustrate the potential competitive advantage to be had from using the numerical range. The first two examples are of hyperbolic matrix polynomials, but the third is not. In each case, it is clear that the roots of the quadratic form are adhering to some spectral region in the plane. Each example contains a plot of the initial estimates using both the numerical range and Newton polygon, as well as the approximated eigenvalues.

\newpage
\begin{ex}[Spring]\label{ex:spring} 
\begin{center}
\includegraphics[width=\textwidth]{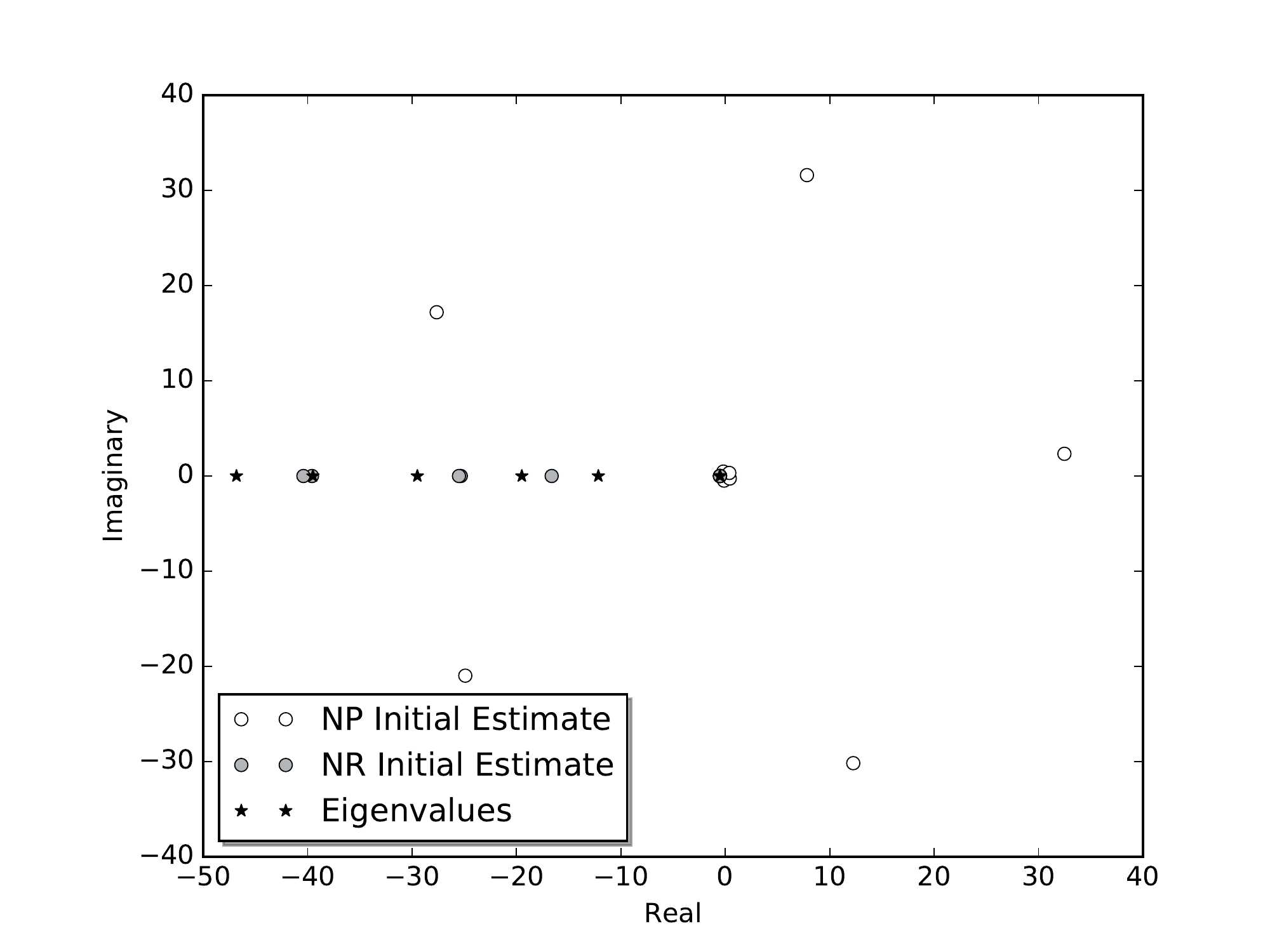}
\end{center}
\end{ex}

\begin{ex}[CD Player]\label{ex:cd_player}
\begin{center}
\includegraphics[width=\textwidth]{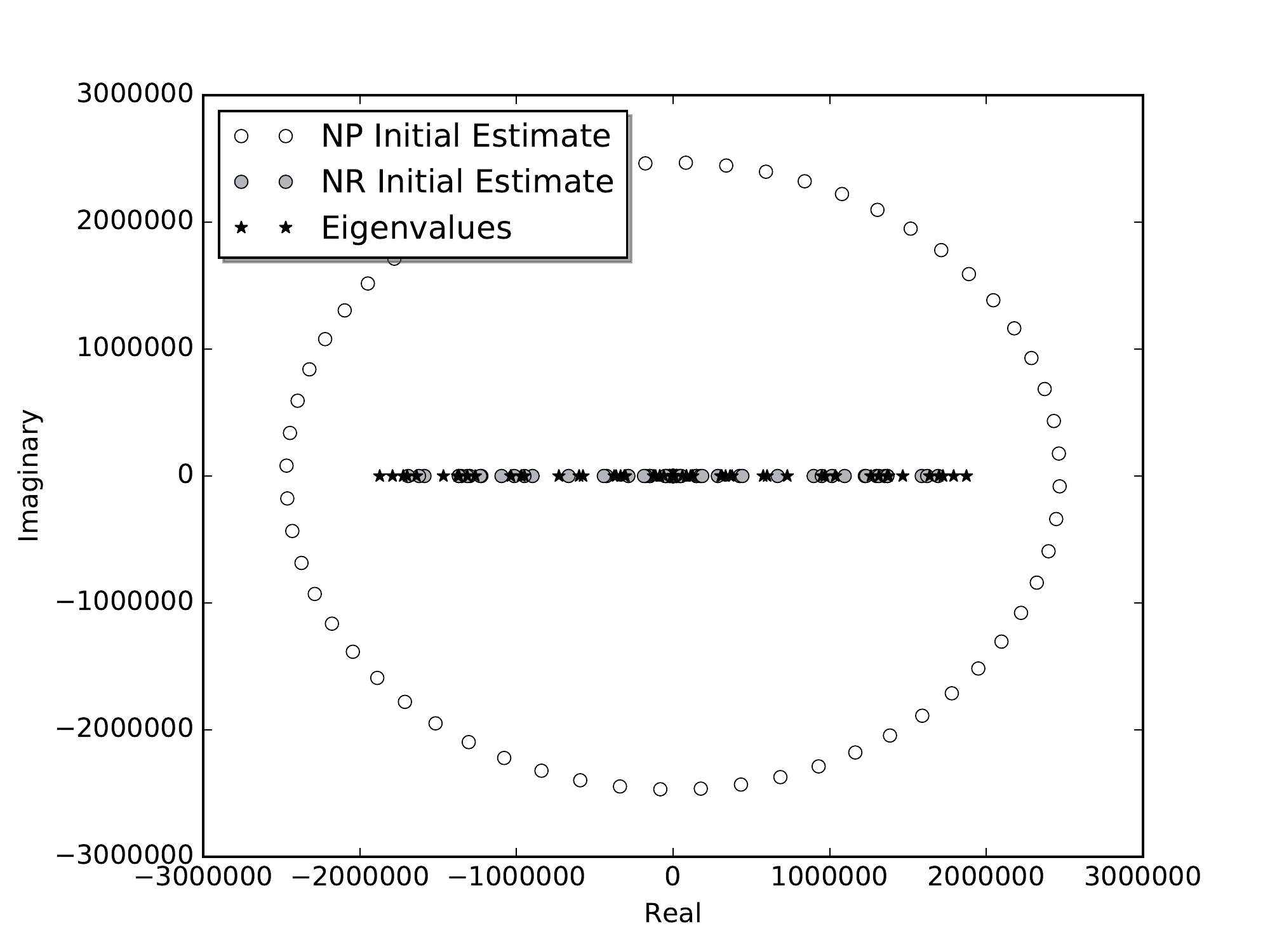}
\end{center}
\end{ex}
The earlier examples highlight the advantage the numerical range has to offer, especially when the eigenvalues are real. This advantage leads to cutting the computation time in half when solving the Spring problem, and by a quarter when solving the CD Player problem. In the following example, the eigenvalues are complex, but the advantage of the numerical range is still evident. Note how the elements from the numerical range clearly identify the $4$ spectral regions in the complex plane.

\begin{ex}[Butterfly]\label{ex:butterfly} 
\begin{center}
\includegraphics[width=\textwidth]{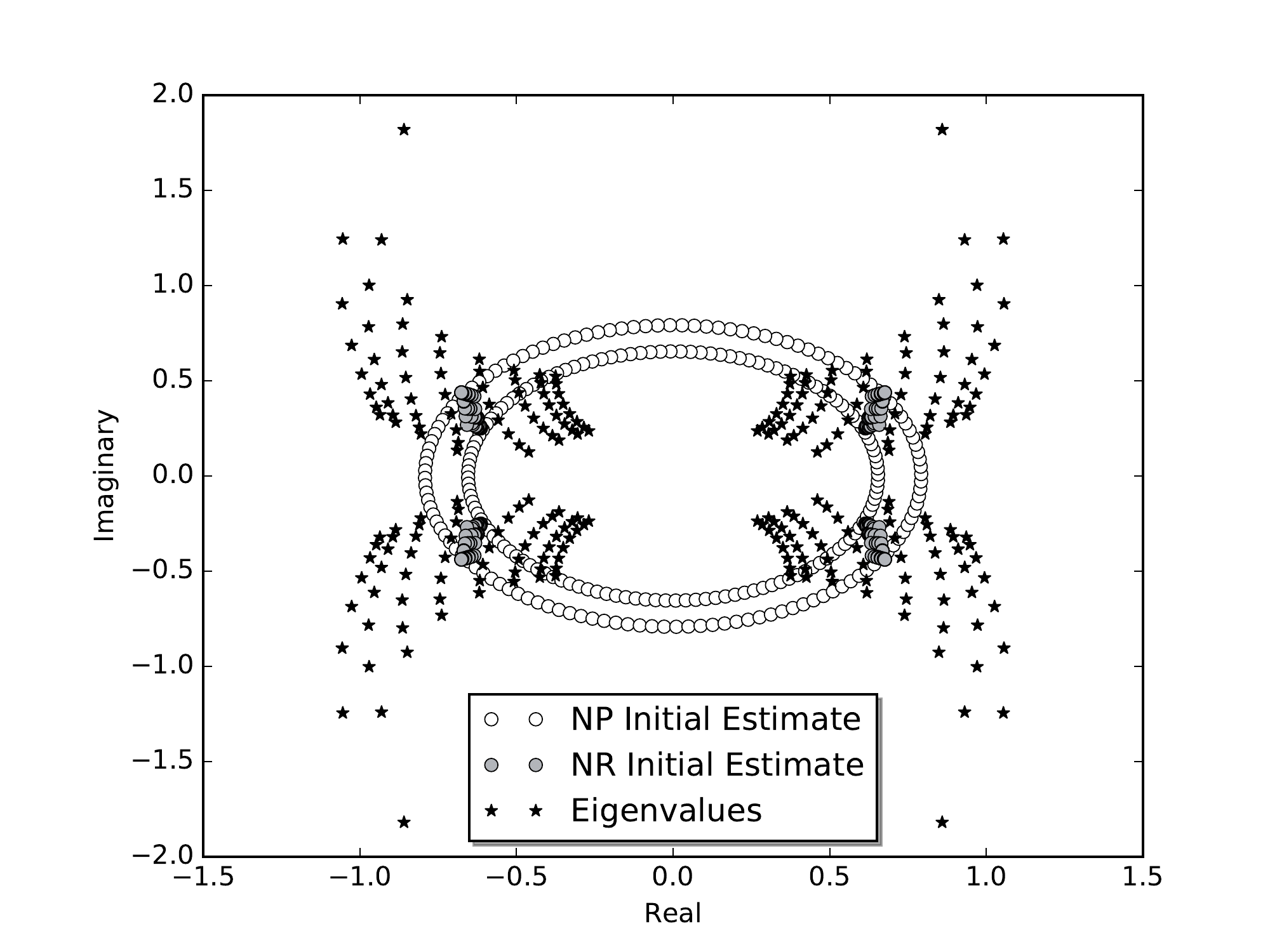}
\end{center}
\end{ex}

Not only do the elements of the numerical range adhere to the spectrum better than points on a circle in the complex plane, they are often, in practice, within the Pellet bounds from Theorem~\ref{thm:Pellet}. We can make the following precise statement.
\begin{theorem}\label{thm:SAMPIE}
Let $P(\lambda)$ be a self-adjoint matrix polynomial. Then for any $\lambda\in W(P)$, $|\lambda|$ is no bigger than the upper Pellet bound.
\end{theorem} 
\begin{proof}
Let $x\in\mathbb{C}^{n}$ be a vector with unit length. The upper Pellet bound on the roots of the polynomial $x^{*}P(\lambda)x$ is the unique real positive solution to the equation
\[
|x^{*}A_{d}x|\mu^{d}=\sum_{i=0}^{d-1}\mu^{i}|x^{*}A_{i}x|.
\]
For any self-adjoint matrix $A$, it is well-know that $\left\Vert A\right\Vert_{2}=\sup\limits_{x^{*}x=1}|x^{*}Ax|$. Therefore,
\[
|x^{*}A_{d}x|\mu^{d}\leq\sum_{i=0}^{d-1}\mu^{i}\left\Vert A_{i}\right\Vert_{2}.
\]
Let $R$ denote the upper bound on the roots of $x^{*}P(\lambda)x$ and $\hat{R}$ denote the upper bound on the eigenvalues of $P(\lambda)$. Then, by Theorem~\ref{thm:Pellet}, $R^{m}\leq\hat{R}^{m}$ and the result follows. 
\end{proof}

\subsubsection{Zero and Infinite Eigenvalues}\label{sec:zeroinf}
Laguerre's method experiences local cubic convergence if the root is simple; otherwise, convergence is linear. In practice, it is most common to have multiple zero and infinite eigenvalues. Therefore to avoid poor performance when dealing with multiple roots, we employ an a priori identification of zero and infinite eigenvalues. During this identification process, we assume that the zero and infinite eigenvalues are semi-simple and thus our problem turns into a familiar one: to determine the rank of the matrices $A_{0}$ and $A_{d}$.

In order to determine the rank of a matrix $A$, we perform a QR factorization with column pivoting. Let $QR=AE$, where \[R=\begin{bmatrix}R_{11} & R_{12}\\ 0 & R_{22}\end{bmatrix}\] $R_{11}$ is $k_{1}\times k_{1}$, $R_{2}$ is $k_{2}\times k_{2}$, $k_{1}+k_{2}=n$, and $E$ is a permutation matrix such that the diagonal entries in $R$ occur in non-increasing order. Our aim is to determine an index $k_{1}$ such that $R_{11}$ is well-conditioned and $R_{22}$ is negligible. If $k_{1}<n$, then the matrix $A$ is rank deficient and the dimension of its null space is $k_{2}$. We compute a basis for the right and left nullspace by
\begin{equation}
x_{j}=E\hat{x}_{j}~\text{  and  }~y_{j}=Qe_{j},
\end{equation}
where \[R(1:k_{1},1:k_{1})\hat{x}_{j}(1:k_{1})=-R(1:k_{1},j),\] $\hat{x}_{j}(k_{1}+1:j-1)=0,~\hat{x}_{j}(j)=1$,~$\hat{x}_{j}(j+1:n)=0$, and $e_{j}$ is the $jth$ standard basis vector for $j=k_{1}+1,\ldots,n$.

The above process is performed on both matrices $A_{0}$ and $A_{d}$, thereby computing the geometric multiplicity of the zero and infinite eigenvalues, respectively, and their corresponding right and left eigenvectors. Once this is done, the columns of the matrix $Q$ are then used to compute initial estimates to the remaining finite eigenvalues via the roots of the quadratic form $q_{j}^{*}P(\lambda)q_{j}$ for $j=1,\ldots,n$. We compute the roots of each polynomial using Laguerre's method, specialized for the scalar polynomial, which was outlined previously. Note that the computation of the QR factorization along with solving the $n$ polynomial equations has a computational complexity of $O(n^{3}+nd^{2})$ and is therefore in accordance with the computational complexity of our method.

\subsection{Hessenberg and Tridiagonal Form}\label{sec:HessTri} 
We are motivated to consider the case where the coefficients of the matrix polynomial are in Hessenberg or tridiagonal form. The Hessenberg case is of both theoretical and practical importance. In light of the original development of Hyman's method, we will consider this method for upper Hessenberg matrix polynomials and note the tridiagonal matrix polynomial as a special case. What's more, every matrix polynomial can be reduced to Hessenberg form~\cite{Cameron2016}. While no numerically stable algorithm currently exists to perform this reduction, there exist applications where this Hessenberg structure arises naturally; for example, the Bilby problem in~\cite{Betcke2013}. With regards to the tridiagonal case, previous developments have focused on the linear and quadratic polynomial eigenvalue problem~\cite{Bini2005, Plestenjak2006}, whereas our development is applicable to any degree polynomial eigenvalue problem.

\subsubsection{Hyman's Method}
Hyman's method, a method for evaluating the characteristic polynomial and its derivatives at a point, is attributed to a conference presentation given by M.A. Hyman of the Naval Ordnance Laboratory in 1957~\cite{Wilkinson1963}. The backward stability of this method has been shown~\cite{Wilkinson1963}, and this method has been used to evaluate the characteristic polynomial of a matrix~\cite{Parlett1964} and matrix pencil~\cite{Gary1965}. Here we generalize these approaches in order to apply Hyman's method to the matrix polynomial.

We denote an upper Hessenberg matrix polynomial as follows \[P(\lambda)=\begin{bmatrix}p_{11}(\lambda) & p_{12}(\lambda) & \cdots & p_{1n}(\lambda)\\ p_{21}(\lambda) & p_{22}(\lambda) & \cdots & p_{2n}(\lambda)\\ & \ddots & \ddots & \vdots\\ & & p_{n,n-1}(\lambda) & p_{nn}(\lambda)\end{bmatrix},\] where $p_{ij}(\lambda)$ is a scalar polynomial of degree at most $d$. Note that in the tridiagonal case $p_{ij}(\lambda)=0$ for $j>i+1$. The insightful observation that Hyman made was that $P(\lambda)$ has the same determinant as \[\begin{bmatrix}p_{11}(\lambda) & p_{12}(\lambda) & \cdots & b(\lambda)\\ p_{21}(\lambda) & p_{22}(\lambda) & \cdots & 0\\ & \ddots & \ddots & \vdots\\ & & p_{n,n-1}(\lambda) & 0\end{bmatrix},\] provided that
\begin{equation}\label{eq:hym1}
P(\lambda)\begin{bmatrix}x_{1}(\lambda)\\\vdots\\x_{n-1}(\lambda)\\1\end{bmatrix}=\begin{bmatrix}b(\lambda)\\0\\\vdots\\0\end{bmatrix}.
\end{equation} If we let $p(\lambda)=\det P(\lambda)$, then \[p(\lambda)=(-1)^{n-1}b(\lambda)q(\lambda),\] where $q(\lambda)=\prod_{j=1}^{n-1}p_{j+1,j}(\lambda)$. Given a fixed scalar $\lambda$, all unknown values in~\eqref{eq:hym1} can be computed in $O(n^{2})$ time for Hessenberg $P(\lambda)$ and in $O(n)$ time for tridiagonal $P(\lambda)$. The values of $x_{1}(\lambda),\ldots,x_{n-1}(\lambda)$ are then used to solve the following equation
\begin{equation}\label{eq:hym2} 
P(\lambda)\begin{bmatrix}x_{1}^{'}(\lambda)\\\vdots\\ x_{n-1}^{'}(\lambda)\\ 0\end{bmatrix}=\begin{bmatrix}b^{'}(\lambda)\\ 0\\\vdots\\ 0\end{bmatrix}-P^{'}(\lambda)\begin{bmatrix}x_{1}(\lambda)\\\vdots\\x_{n-1}(\lambda)\\1\end{bmatrix}.
\end{equation}
Then the values of $x_{1}(\lambda),\ldots,x_{n-1}(\lambda)$ and their derivatives are used to compute $b^{''}(\lambda)$
\begin{equation}\label{eq:hym3}
P(\lambda)\begin{bmatrix}x_{1}^{''}(\lambda)\\\vdots\\ x_{n-1}^{''}(\lambda)\\ 0\end{bmatrix}=\begin{bmatrix}b^{''}(\lambda)\\ 0\\\vdots\\ 0\end{bmatrix}-2P^{'}(\lambda)\begin{bmatrix}x_{1}^{'}(\lambda)\\\vdots\\x_{n-1}^{'}(\lambda)\\0\end{bmatrix}-P^{''}(\lambda)\begin{bmatrix}x_{1}(\lambda)\\\vdots\\x_{n-1}(\lambda)\\1\end{bmatrix}.
\end{equation}

Once $b(\lambda)$, $b^{'}(\lambda)$, and $b^{''}(\lambda)$ have been computed, an efficient computation of the Laguerre correction term can be obtained from the following
\begin{equation}\label{eq:hessclit}
\begin{split}
\frac{p^{'}(\lambda)}{p(\lambda)}&=\frac{b^{'}(\lambda)+b(\lambda)\left(\frac{q^{'}(\lambda)}{q(\lambda)}\right)}{b(\lambda)},\\
-\left(\frac{p^{'}(\lambda)}{p(\lambda)}\right)^{'}&=\left(\frac{p^{'}(\lambda)}{p(\lambda)}\right)^{2}-\left(\frac{b^{''}(\lambda)+2b^{'}(\lambda)\left(\frac{q^{'}(\lambda)}{q(\lambda)}\right)+b(\lambda)\left(\frac{q^{''}(\lambda)}{q(\lambda)}\right)}{b(\lambda)}\right).
\end{split}
\end{equation}
Note that we have carefully avoided the potentially hazardous product in computing $q(\lambda)$ and its derivatives by replacing it with 
$$
\frac{q^{'}(\lambda)}{q(\lambda)}=\underset{j=1}{\overset{n-1}{\sum}}\frac{p_{j+1,j}^{'}(\lambda)}{p_{j+1,j}(\lambda)},
$$
and
$$
\frac{q^{''}(\lambda)}{q(\lambda)}=\left(\frac{q^{'}(\lambda)}{q(\lambda)}\right)^{'}+\left(\frac{q^{'}(\lambda)}{q(\lambda)}\right)^{2},
$$
where $\left(\frac{q^{'}(\lambda)}{q(\lambda)}\right)^{'}=\underset{j=1}{\overset{n-1}{\sum}}\left(\frac{p_{j+1,j}^{''}(\lambda)}{p_{j+1,j}(\lambda)}-\left(\frac{p_{j+1,j}^{'}(\lambda)}{p_{j+1,j}(\lambda)}\right)^{2}\right)$.

Several remarks are in order. First, if any subdiagonal of $P(\lambda)$ is zero, then solving~\eqref{eq:hym1}-\eqref{eq:hym3} will require division by zero. Fortunately, we can replace any zero subdiagonal with double precision unit roundoff $\epsilon$ and maintain the backward stability of Hyman's method~\cite{Wilkinson1963}. Second Hyman's method significantly reduces the cost of each iteration and the resulting cost of our method is $O(dn^{3}+d^{2}n^{3})$ for Hessenberg matrix polynomials and $O(d^{2}n^{2})$ for tridiagonal matrix polynomials.

\subsubsection{Eigenvectors, Stopping Criteria, and Condition Numbers}\label{sec:HessESC}
Let $\lambda\in\mathbb{C}$ be an approximate eigenvalue and define
\begin{equation}\label{eq:hessqrf}
QR=P(\lambda).
\end{equation}
This factorization can be done in $O(n^{2})$ time for Hessenberg $P(\lambda)$ and in $O(n)$ time for tridiagonal $P(\lambda)$. Let $j$ denote the index that minimizes $|r_{jj}|$. If $|r_{jj}|<\tau$, where $\tau$ is some predetermined tolerance, then we say that the approximate eigenvalue has converged. This constitutes our first stopping criterion. Given that $\lambda$ has converged, we compute the corresponding right and left eigenvectors using
\begin{equation}\label{eq:hessev1}
x=\hat{x}~\text{ and }~y=Q\hat{y},
\end{equation}
where 
\begin{gather*}
R(1:j-1,1:j-1)\hat{x}(1:j-1)=-R(1:j-1,j),\\
R(j+1:n,j+1:n)\hat{y}(j+1:n)=-R(j+1:n,j),
\end{gather*}
$\hat{x}(j)=1$,~$\hat{x}(j+1:n)=0$,~$\hat{y}(j)=1$, and $\hat{y}(1:j-1)=0$.

If there exists no index $j$ such that $|r_{jj}|<\tau$, then we compute an upper bound for the backward error of the eigenvalue approximation via~\eqref{eq:berrUB}. If the backward error of $\lambda$ is less than $\epsilon$, then we say that $\lambda$ has converged. This constitutes our second stopping criterion. We then apply inverse iteration to 
\[
R^{*}R~\text{ and }~QRR^{*}Q^{*},
\]
to compute the right and left singular vectors, respectively. Using~\eqref{eq:hessev1} to form initial estimates for the inverse iteration results in quick convergence to excellent eigenvector approximations. 

As was done in \S\ref{sec:ESC}, we also check if the approximate eigenvalue $\lambda$ satisfies~\eqref{eq:ssc}. In this case, no significant change to the current eigenvalue approximation is made, and we say that $\lambda$ has converged. This constitutes our third stopping criterion. 

In summary, given an approximate eigenvalue, we compute the QR factorization in~\eqref{eq:hessqrf}. If any of the three stopping criteria are met, or the maximum number of iterations allowed is reached, then we cease to update the eigenvalue approximation and compute corresponding right and left eigenvectors. Otherwise, we use Hyman's method to compute the Laguerre iterate. Once the approximate eigenvalue has converged and the corresponding right and left eigenvectors are computed, we report each eigenvalue's condition number~\eqref{eq:con}.

\subsubsection{Initial Estimates}
Just as was done with the general matrix polynomial, initial estimates consist of computing the geometric multiplicity of the zero and infinite eigenvalues, a basis for the corresponding eigenspace, and initial estimates to the remaining finite eigenvalues via the numerical range. For the Hessenberg case, there is no difference whatsoever, since we can accomplish all of the above while adhering to the cost of the method. However, for the tridiagonal case we must make several changes in order to align with the method's cost. 

When computing the geometric multiplicity of the zero and infinite eigenvalues, we must settle for only a QR factorization of the coefficient matrices $A_{0}$ and $A_{d}$, since the column pivoting has the potential to destroy the tridiagonal structure and make this method too expensive. Therefore, we cannot expect that the diagonal entries of the upper triangular $R$ appear in descending order. It is for this reason that we identify the pivots of $R$ one row at a time. By keeping track of the location of the previous pivot and utilizing the structure of $R$, we can identify whether or not each row has a pivot, and the location of said pivot, in $O(n^{2})$ time. Then, the dimension of the corresponding eigenspace is $(n-k)$, where $k$ is the number of rows with a pivot. If $(n-k)>1$, then we use the location of each non-pivot column to compute a basis for eigenspace in $O(n^{2})$ time. 

The QR factorization of the tridiagonal matrices $A_{0}$ and $A_{d}$ is computed using plane rotations and therefore each column vector of $Q$ can be computed in $O(n)$ time. Furthermore, each quadratic form $x^{*}P(\lambda)x$ can be computed in $O(dn)$ time and the roots of each scalar polynomial can be computed in $O(d^{2})$ time. It follows that the initial estimates of the tridiagonal matrix polynomial can be found in $O(d^{2}n^{2})$ time. 

\section{Stability}\label{sec:stability}
The stability of any numerical method is of the utmost importance. In this section, we provide a detailed account of why our method is robust against the potentially harmful overflow in the evaluation of the matrix polynomial and its derivatives. Furthermore, we identify the predetermined tolerance used in the stopping criteria (\S\ref{sec:ESC}) and show that if either the first or second stopping criterion holds then we can guarantee the backward stability of our eigenvalue approximation.  

\subsection{Robustness against Overflow}\label{sec:rao}
Both the computation of the Laguerre iterate as well as the corresponding eigenvector approximation is driven by a QR factorization of $P(\lambda)$, with column pivoting for the general matrix polynomial, where $\lambda$ is the current eigenvalue approximation. The evaluation of $P(\lambda)$ can be done efficiently using Horner's method, but for large degree matrix polynomials this computation is prone to overflow. It is for this reason that when $|\lambda|>1$ we opt to work with the reversal polynomial~\eqref{eq:revpoly}, with $\rho=1/\lambda$. One may argue that now our computation is prone to underflow, but this is not harmful; as $\lambda\rightarrow\infty$, $\rP(\rho)\rightarrow A_{d}$, which is aligned with our definition of the infinite eigenvalues of $P(\lambda)$ being the zero eigenvalues of $\rP(\rho)$. 

Now, the general Laguerre correction term in~\eqref{eq:clit} becomes:
\begin{equation}\label{eq:crlit}
\begin{split}
\frac{p^{'}(\lambda)}{p(\lambda)}&=\rho\cdot\text{trace}\left(dI-\rho X_{3}(\rho)\right),\\
-\left(\frac{p^{'}(\lambda)}{p(\lambda)}\right)^{'}&=\rho^{2}\cdot\text{trace}\left(dI-2\rho X_{3}(\rho)+\rho^{2}\left(X_{3}^{2}(\rho)-X_{4}(\rho)\right)\right),
\end{split}
\end{equation}
where $\rP(\rho)X_{3}(\rho)=\rP^{'}(\rho)$, and $\rP(\rho)X_{4}(\rho)=\rP^{''}(\rho)$. By using~\eqref{eq:clit} when $|\lambda|\leq 1$ and~\eqref{eq:crlit} when $|\lambda|>1$, we have a method for computing the Laguerre iterate of a matrix polynomial which is robust against overflow. This is similar to the approach in~\cite{Bini1996} for evaluating polynomials, but to our knowledge, we are the first to apply this to matrix polynomials. 

For Hessenberg matrix polynomials (tridiagonal case included), we apply Hyman's method to the reversal polynomial in order to obtain the values of $r^{'}(\rho)/r(\rho)$ and $r^{''}(\rho)/r(\rho)$, where $r(\rho)=\det\rP(\rho)$. The Laguerre correction term in~\eqref{eq:hessclit} then becomes:
\begin{equation}\label{eq:hesscrlit}
\begin{split}
\frac{p^{'}(\lambda)}{p(\lambda)}&=\rho\left(nd-\rho\frac{r^{'}(\rho)}{r(\rho)}\right),\\
-\left(\frac{p^{'}(\lambda)}{p(\lambda)}\right)^{'}&=\rho^{2}\left(nd-2\rho\frac{r^{'}(\rho)}{r(\rho)}+\rho^{2}\left(\left(\frac{r^{'}(\rho)}{r(\rho)}\right)^{2}-\frac{r^{''}(\rho)}{r(\rho)}\right)\right).
\end{split}
\end{equation}
By using~\eqref{eq:hessclit} when $|\lambda|\leq 1$ and~\eqref{eq:hesscrlit} when $|\lambda|>1$, we have a method for computing the Laguerre iteration of an upper Hessenberg matrix polynomial (tridiagonal case included) which is both efficient and robust against overflow. 

For nonzero eigenvalues $\ker P(\lambda)=\ker\rP(\rho)$, where $\rho=1/\lambda$. Therefore, if $|\lambda|>1$ then we switch $P(\lambda)$ with $\rP(\rho)$ in both~\eqref{eq:qrf}, for the general matrix polynomial, and~\eqref{eq:hessqrf}, for the Hessenberg matrix polynomial (which includes the tridiagonal case). The discussion on computing corresponding right and left eigenvectors in \S\ref{sec:ESC}, for the general matrix polynomial, and \S\ref{sec:HessESC}, for the Hessenberg matrix polynomial, carries over naturally. With the exception that the upper bound on the backward error in the eigenvalue approximation from~\eqref{eq:berrUB} becomes:
\begin{equation}
\frac{\left\Vert b \right\Vert_{2}}{\ra\left\Vert \rP(\rho)^{-1}b \right\Vert_{2}},
\end{equation}
where $\ra=\sum_{i=0}^{d}|\rho|^{d-i}\left\Vert A_{i} \right\Vert_{2}$, and $b\in\mathbb{C}^{n}$ is nonzero. In addition, the normwise condition number from~\eqref{eq:con} becomes:
\begin{equation}
\kappa(\lambda,P)=\frac{\ra\left\Vert x \right\Vert_{2}\left\Vert y \right\Vert_{2}}{|y^{*}(d\cdot\rP(\rho)-\rho\rP^{'}(\rho))x|}.
\end{equation}

\subsection{Backward Stability}\label{sec:bs}
Let $\lambda\in\mathbb{C}$ be an approximate eigenvalue, $x$ a corresponding right eigenvector, and $y$ a corresponding left eigenvector. Following the development in~\cite{Tisseur2000}, we define the normwise backward error of the right eigenpair by
\begin{equation}\label{eq:berrdef}
\eta(\lambda,x)=\min\{\epsilon\colon~[P(\lambda)+\Delta P(\lambda)]x=0,\,\left\Vert\Delta A_{i}\right\Vert_{2}\leq\epsilon\left\Vert A_{i} \right\Vert_{2},\,i=0,1,\ldots,d\}
\end{equation}
where $\Delta P(\lambda)=\sum_{i=0}^{d}\lambda^{i}\Delta A_{i}$. This definition of the normwise backward error is concerned with a relative measurement of perturbation in the coefficients of the matrix polynomial. The normwise backward error for the left eigenpair $(\lambda, y)$ is similarly defined.

The first stopping criterion outlined in Sections~\ref{sec:ESC} and~\ref{sec:HessESC} is concerned with the smallest diagonal entry of $R$ being less than $\tau$. We define $\tau=\alpha\epsilon$ if $|\lambda|\leq 1$ and $\tau=\ra\epsilon$ otherwise, where $\alpha=\sum_{i=0}^{d}|\lambda|^{i}\left\Vert A_{i} \right\Vert_{2}$, $\ra=\sum_{i=0}^{d}|\rho|^{d-i}\left\Vert A_{i} \right\Vert_{2}$, $\rho=1/\lambda$, and $\epsilon$ is double precision unit roundoff. 
\begin{theorem}\label{thm:berr1}
If the first stopping criterion holds, then the approximate right eigenpair has a backward error bounded above by $\epsilon(2n+1)+O(\epsilon^{2})$.
\end{theorem}
\begin{proof}
From definition~\eqref{eq:berrdef} and~\cite{Tisseur2000}[Theorem 1] it follows that we may compute the normwise backward error for the right eigenpair $(\lambda,x)$ by
\begin{equation}\label{eq:berr1}
\begin{cases}
\eta(\lambda,x)=\frac{\left\Vert P(\lambda)x \right\Vert_{2}}{\alpha\left\Vert x \right\Vert_{2}} & \text{if $|\lambda|\leq 1$,} \\
\eta(\lambda,x)=\frac{\left\Vert \rP(\rho)x \right\Vert_{2}}{\ra\left\Vert x \right\Vert_{2}} & \text{otherwise}.
\end{cases}
\end{equation}
Without loss of generality we assume that $|\lambda|\leq 1$ for the remainder of the proof. Denote by $QR$ the QR factorization, with column pivoting for general matrix polynomials, of $P(\lambda)$. Denote by $x$ the corresponding eigenvector, for the general matrix polynomial see~\eqref{eq:ev1} and for the Hessenberg matrix polynomial (including tridiagonal case) see~\eqref{eq:hessev1}. Then the computed right eigenvector satisfies
\[
(R+\delta R)x=b,
\]
where $\left\Vert b\right\Vert_{2}<\tau$. It follows from~\cite{Watkins2010}[Corollary 2.7.9] that $\left\Vert\delta R\right\Vert_{F}\leq 2n\epsilon\left\Vert R\right\Vert_{F}+O(\epsilon^{2})$, where $F$ denotes the Frobenius norm. Therefore, 
\[
\left\Vert Rx\right\Vert_{2}\leq\tau + \left\Vert\delta R\right\Vert_{F}\left\Vert x\right\Vert_{2}. 
\]
Recall, in practice, that we replace the matrix 2-norm in the definition of $\alpha$ with the Frobenius norm and note that $\left\Vert R\right\Vert_{F}\leq\alpha$. Thus, the result follows from dividing both sides of the above equation by $\alpha\left\Vert x\right\Vert_{2}$ to give
\[
\frac{\left\Vert Rx\right\Vert_{2}}{\alpha\left\Vert x\right\Vert_{2}}\leq\epsilon(1 + 2n)+O(\epsilon^{2}).
\]
\end{proof}

If an approximate eigenvector has not been computed, then an appropriate measure of the backward error is given by
\begin{equation}\label{eq:berr2}
\begin{cases}
\eta(\lambda)=\frac{1}{\alpha\left\Vert P(\lambda)^{-1}\right\Vert_{2}} & \text{if $|\lambda|\leq 1$,} \\
\eta(\lambda)=\frac{1}{\ra\left\Vert rP(\rho)^{-1}\right\Vert_{2}} & \text{otherwise}.
\end{cases}
\end{equation}
Again, without loss of generality, we assume that $|\lambda|\leq 1$.

If $\lambda$ is an approximate eigenvalue for which the second stopping criterion holds, then there exists a nonzero vector $b\in\mathbb{C}^{n}$ such that 
\[
\frac{\left\Vert b\right\Vert_{2}}{\alpha\left\Vert P(\lambda)^{-1}\right\Vert_{2}}<\epsilon,
\]
it follows that the backward error in the approximate eigenvalue~\eqref{eq:berr2} is bounded above by $\epsilon$. The corresponding right eigenvector is computed as an approximate right singular vector of $P(\lambda)$ corresponding to the smallest singular value $\lambda$ and therefore minimizes the backward error in the right eigenpair~\eqref{eq:berr1}.

Note that the results in the section hold naturally for left eigenpairs. Additionally, the result in Theorem~\ref{thm:berr1} is a worst case scenario and typically you can ignore the factor of $n$. Finally, even though we can only guarantee the backward stability of our eigenvalue approximation if the first or second stopping criterion hold, in practice it is highly unlikely to experience anything but backward stability.  

\section{Numerical Experiments}\label{sec:NE}
We have implemented the algorithm for solving the polynomial eigenvalue problem via Laguerre's method in the software package LMPEP. This package contains our implementation in FORTRAN 90 and can be freely downloaded from Github by visiting \url{https://github.com/Nick314159/LMPEP}. 

In this section, we provide numerical experiments to verify the computational complexity, stability, and accuracy of our methods. All tests were performed on a computer running CENTOS 7 with an Intel Core i5 processor, where the code was compiled with the GNU Fortran (GCC) 4.8.5 20150623 (Red Hat 4.8.5-11) compiler.

\subsection{Complexity}
We first verify the asymptotic complexity of the method. In \S\ref{sec:LMPEP} it was shown that the computational complexity of the method for general matrix polynomials is $O(dn^{4}+d^{2}n^{3})$, and therefore for scalar polynomials, the expected computational complexity is $O(d^{2})$. In addition, in \S\ref{sec:HessTri} it was shown that the computational complexity of the method for Hessenberg matrix polynomials is $O(dn^{3}+d^{2}n^{3})$ and for tridiagonal matrix polynomials is $O(dn^{2}+d^{2}n^{2})$. Four tests were executed:
\begin{itemize}
\item	For the general matrix polynomial, we verify the quadratic complexity in $d$ by fixing $n=2$ and computing the eigenvalues of random matrix polynomials of degree $d=50, 100, \ldots, 1600$. We also verify the quartic complexity in $n$ by fixing $d=2$ and computing the eigenvalues of random matrix polynomials of size $n=20, 40, \ldots, 320$.
\item	For the scalar polynomial, we verify the quadratic complexity by computing the roots of random polynomials of degree $d=50, 100, \ldots, 6400$. We compare the timings with the POLZEROS from~\cite{Bini1996} and AMVW from~\cite{Aurentz2015}. 
\item	For the Hessenberg matrix polynomial, we verify the quadratic complexity in $d$ by fixing $n=2$ and computing the eigenvalues of random Hessenberg matrix polynomials of degree $d=50, 100, \ldots, 1600$. We also verify the cubic complexity in $n$ by fixing $d=2$ and computing eigenvalues of random Hessenberg matrix polynomials of size $n=20, 40, \ldots, 320$.
\item	For the tridiagonal matrix polynomial, we verify the quadratic complexity in $d$ by fixing $n=2$ and computing the eigenvalues of random tridiagonal matrix polynomials of degree $d=50, 100, \ldots, 1600$. We also verify the quadratic complexity in $n$ by fixing $d=2$ and computing eigenvalues of random tridiagonal matrix polynomials of size $n=20, 40, \ldots, 320$.
\end{itemize}

\begin{minipage}{\linewidth}
\centering
\begin{minipage}{0.5\textwidth}
\centering
\includegraphics[width=\textwidth]{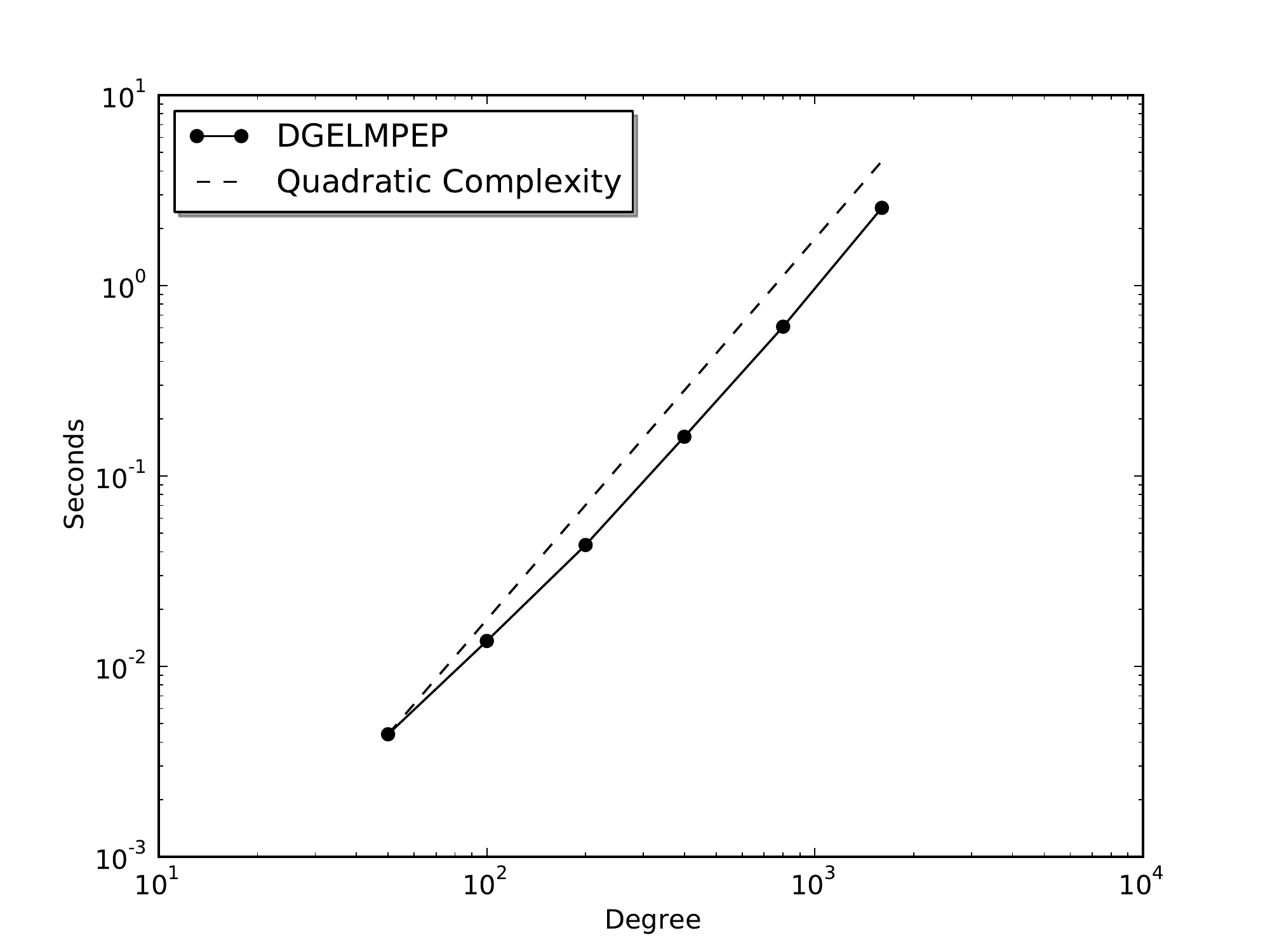}
\end{minipage}\hfill
\begin{minipage}{0.5\textwidth}
\centering
\includegraphics[width=\textwidth]{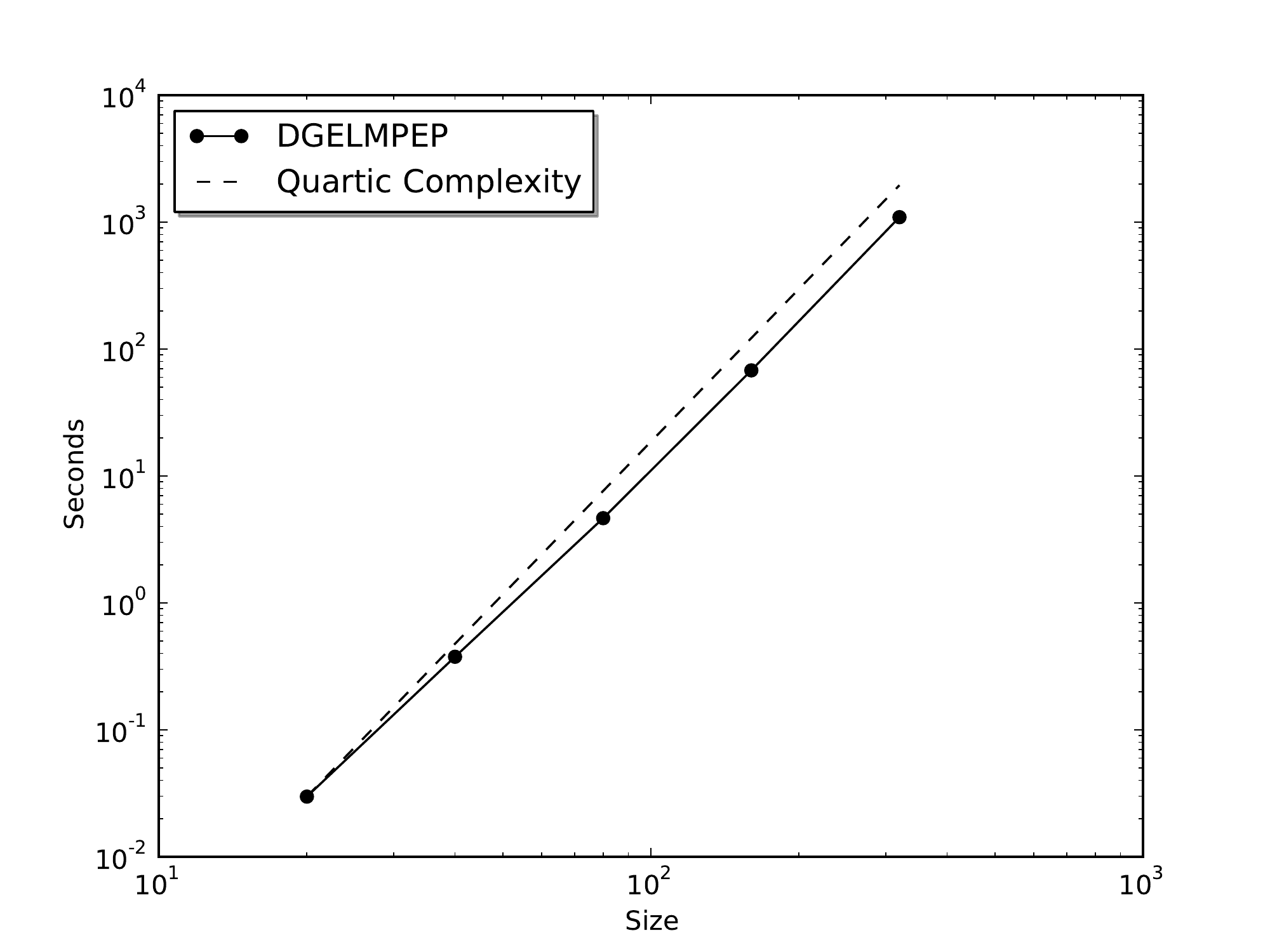}
\end{minipage}
\captionof{figure}{Test of the quadratic complexity of degree $d$ and quartic complexity of size $n$ of the general matrix polynomial. The tests are averaged over $5$ runs.}
\end{minipage}

\begin{minipage}{\linewidth}
\centering
\includegraphics[width=0.6\textwidth]{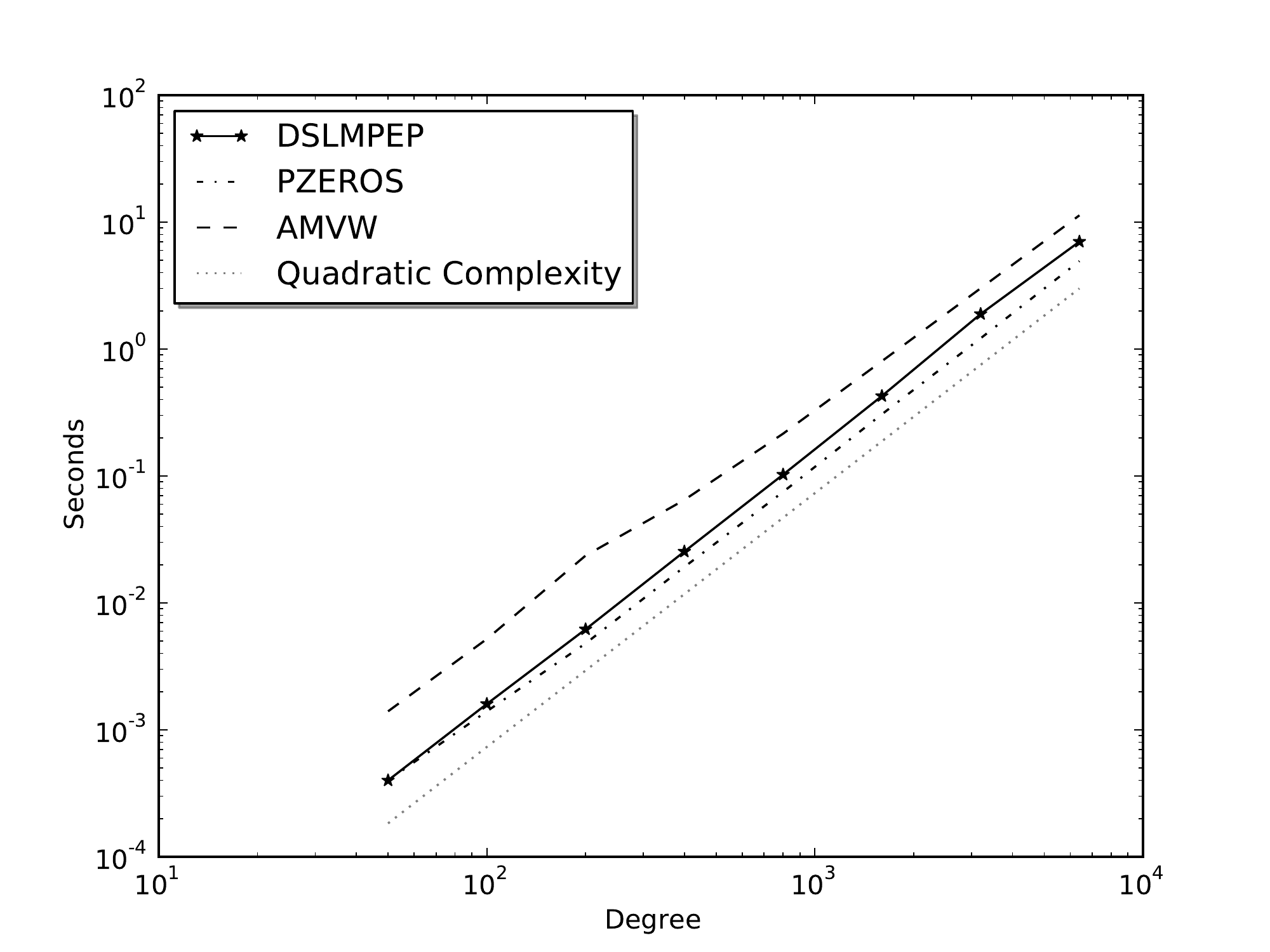}
\captionof{figure}{Test of the quadratic complexity of degree $d$ of the scalar polynomial. The tests are averaged over $5$ runs and runtimes are reported for POLZEROS and AMVW.}
\end{minipage}

\begin{minipage}{\linewidth}
\centering
\begin{minipage}{0.5\textwidth}
\centering
\includegraphics[width=\textwidth]{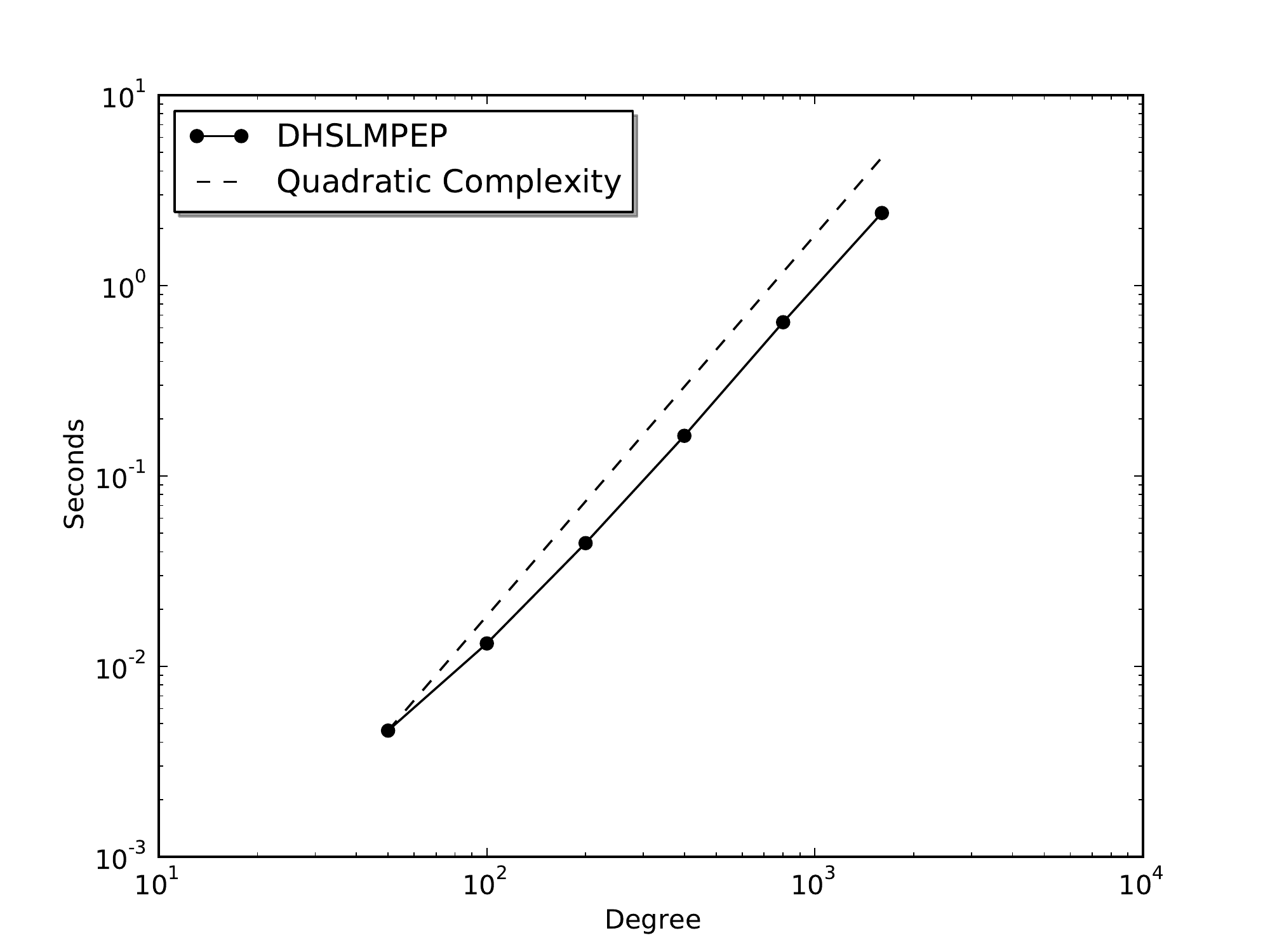}
\end{minipage}\hfill
\begin{minipage}{0.5\textwidth}
\centering
\includegraphics[width=\textwidth]{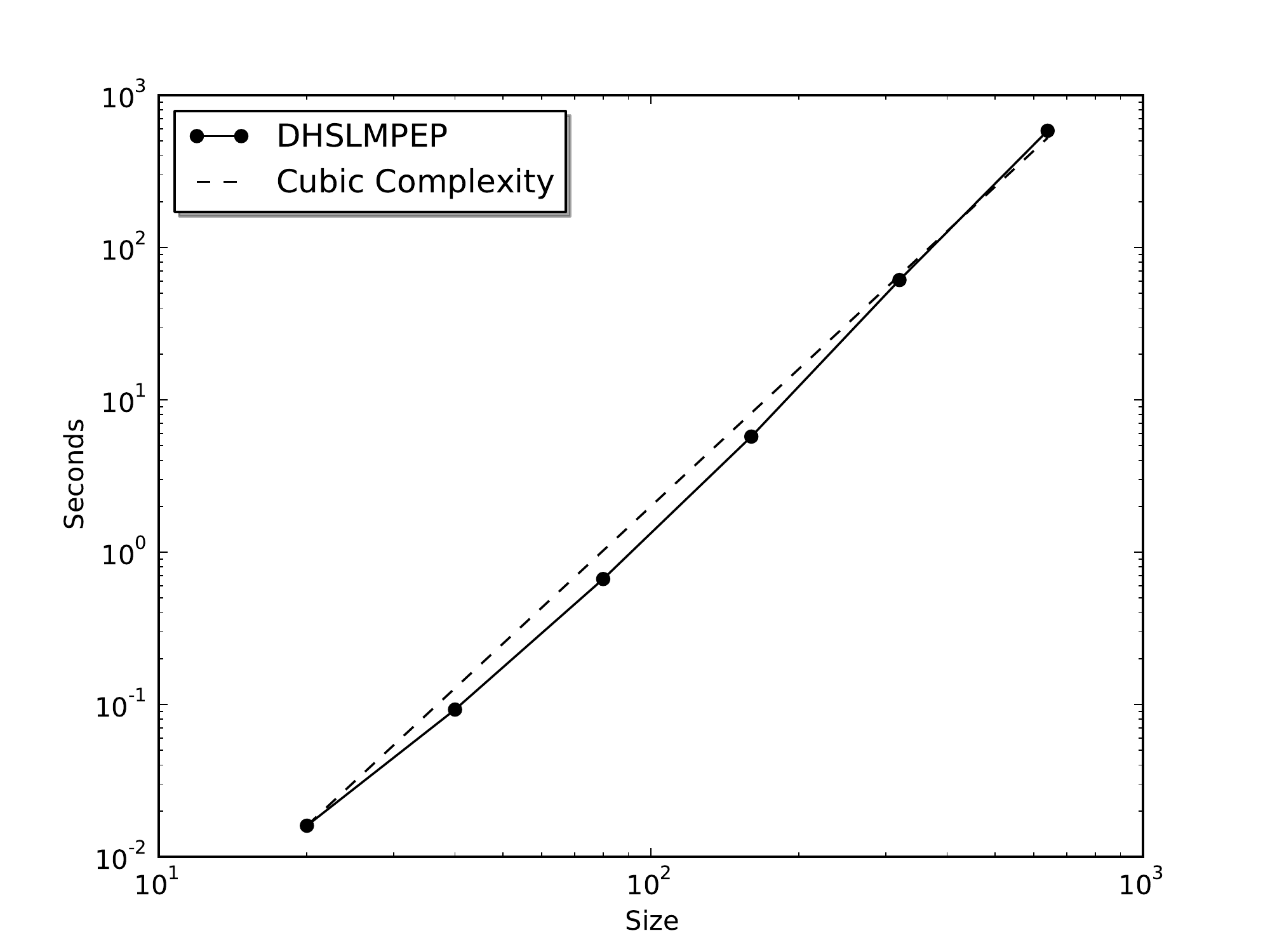}
\end{minipage}
\captionof{figure}{Test of the quadratic complexity of degree $d$ and cubic complexity of size $n$ of the Hessenberg matrix polynomial. The tests are averaged over $5$ runs.}
\end{minipage}

\begin{minipage}{\linewidth}
\centering
\begin{minipage}{0.5\textwidth}
\centering
\includegraphics[width=\textwidth]{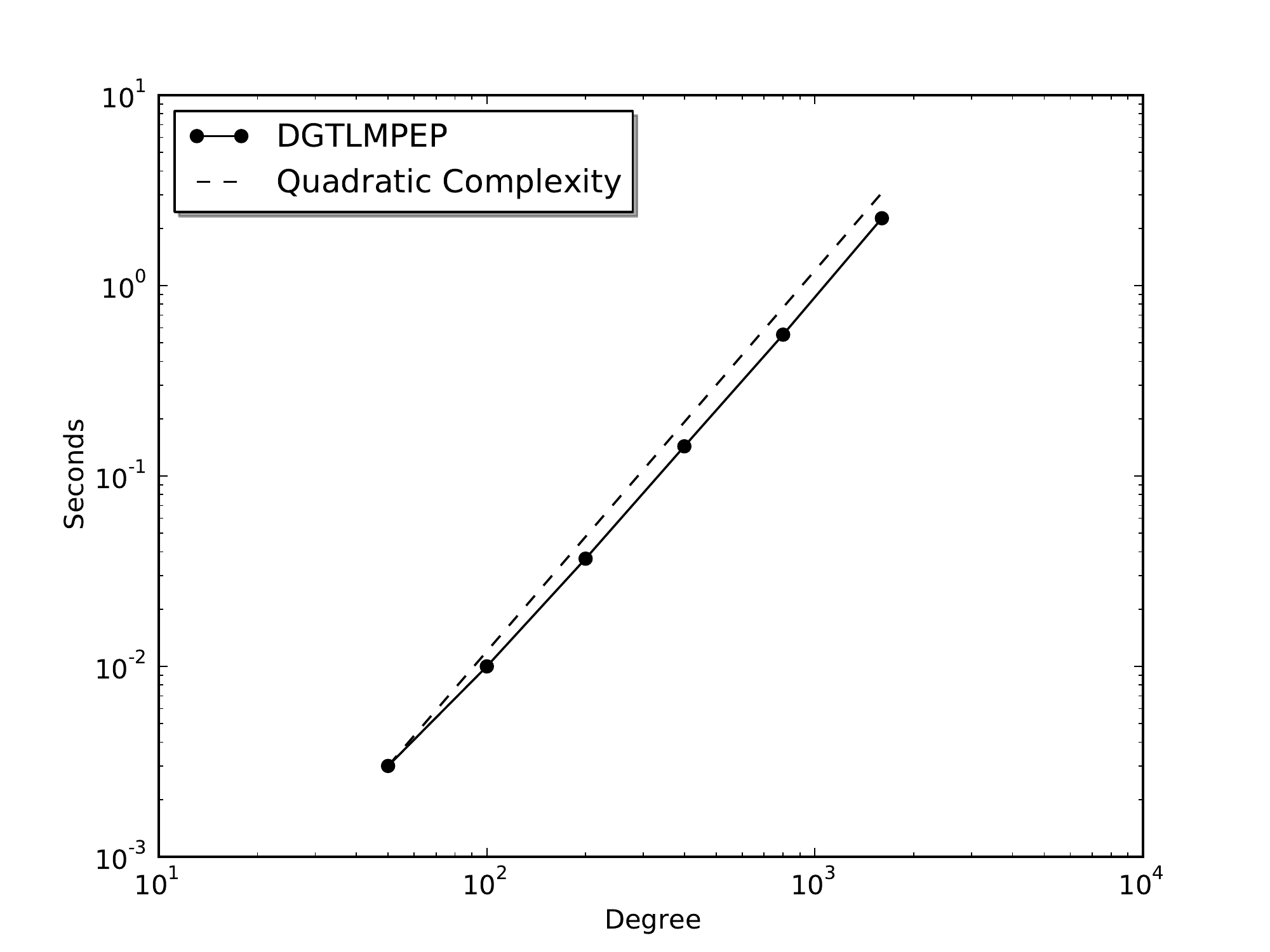}
\end{minipage}\hfill
\begin{minipage}{0.5\textwidth}
\centering
\includegraphics[width=\textwidth]{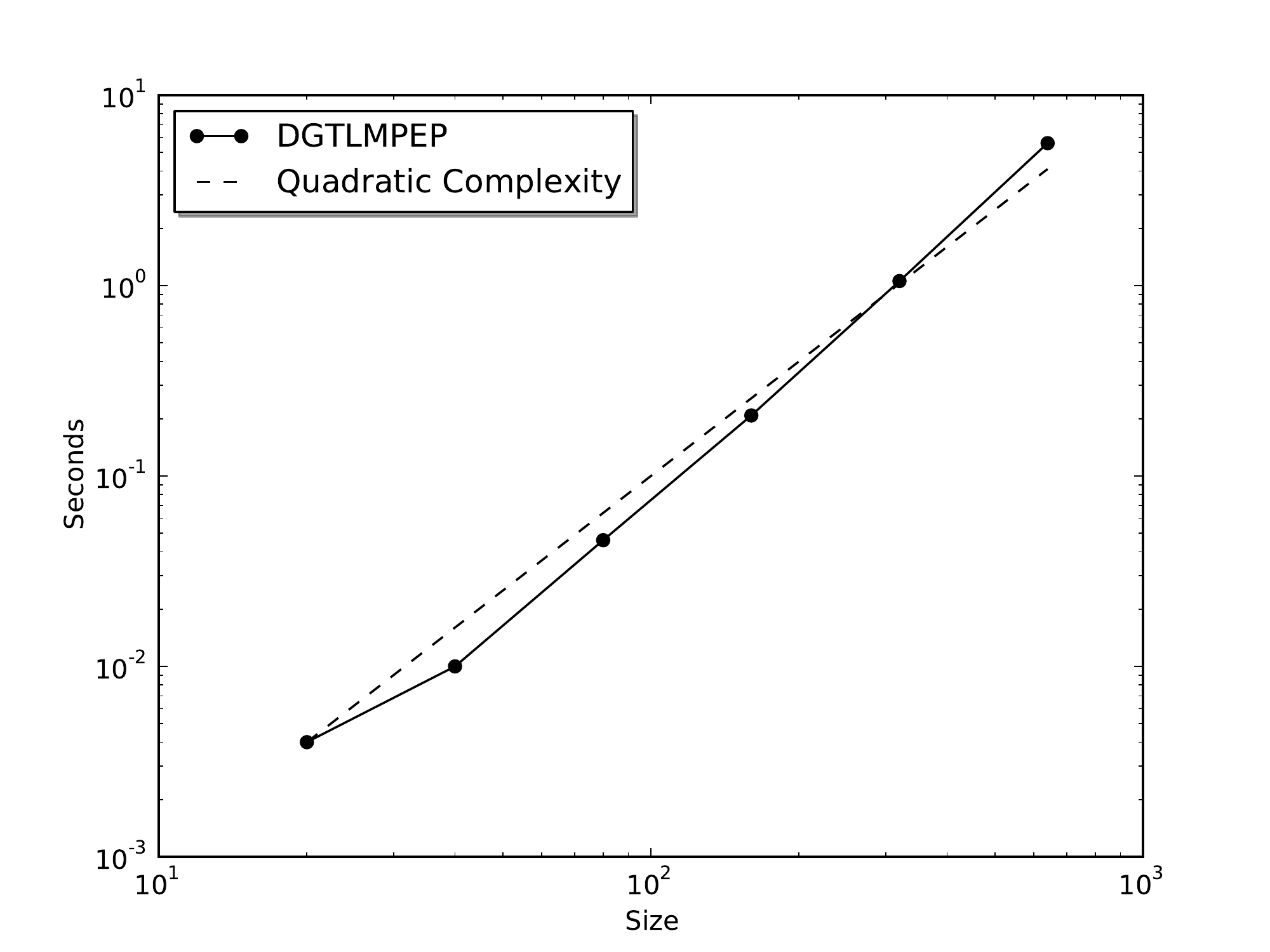}
\end{minipage}
\captionof{figure}{Test of the quadratic complexity of degree $d$ and quadratic complexity of size $n$ of the tridiagonal matrix polynomial. The tests are averaged over $5$ runs.}
\end{minipage}

\subsection{Stability and Accuracy}
We then verify the stability of our method. In \S\ref{sec:stability} it was shown that our method is robust against overflow and that if either the first or second stopping criterion are met then the approximate eigenpair has a tiny backward error. This along with a well-conditioned problem implies that our method is highly accurate. Four tests were executed: 
\begin{itemize}
\item   For the scalar polynomial, we verify the accuracy of our method by computing the roots of random polynomials of degree $d=50,100,\ldots,6400$. We compare the forward error with POLZEROS from~\cite{Bini1996} and AMVW from~\cite{Aurentz2015}.
\item   For the tridiagonal matrix polynomial, we verify the accuracy of our method by computing the eigenvalues of selected problems from both~\cite{Bini2005} and~\cite{Plestenjak2006}. The forward error in our method is compared to the forward error in each respective method. 
\item   For the general matrix polynomial, we verify the stability of our method by solving select problems from the NLEVP package~\cite{Betcke2013} and comparing the backward error in our approximation to the backward error in QUADEIG.
\item   For the general matrix polynomial, we verify the accuracy of our method by comparing the forward error in our approximations to those from QUADEIG for select problems from the NLEVP package~\cite{Betcke2013}. 
\end{itemize}

\begin{minipage}{\linewidth}
\centering
\includegraphics[width=0.75\textwidth]{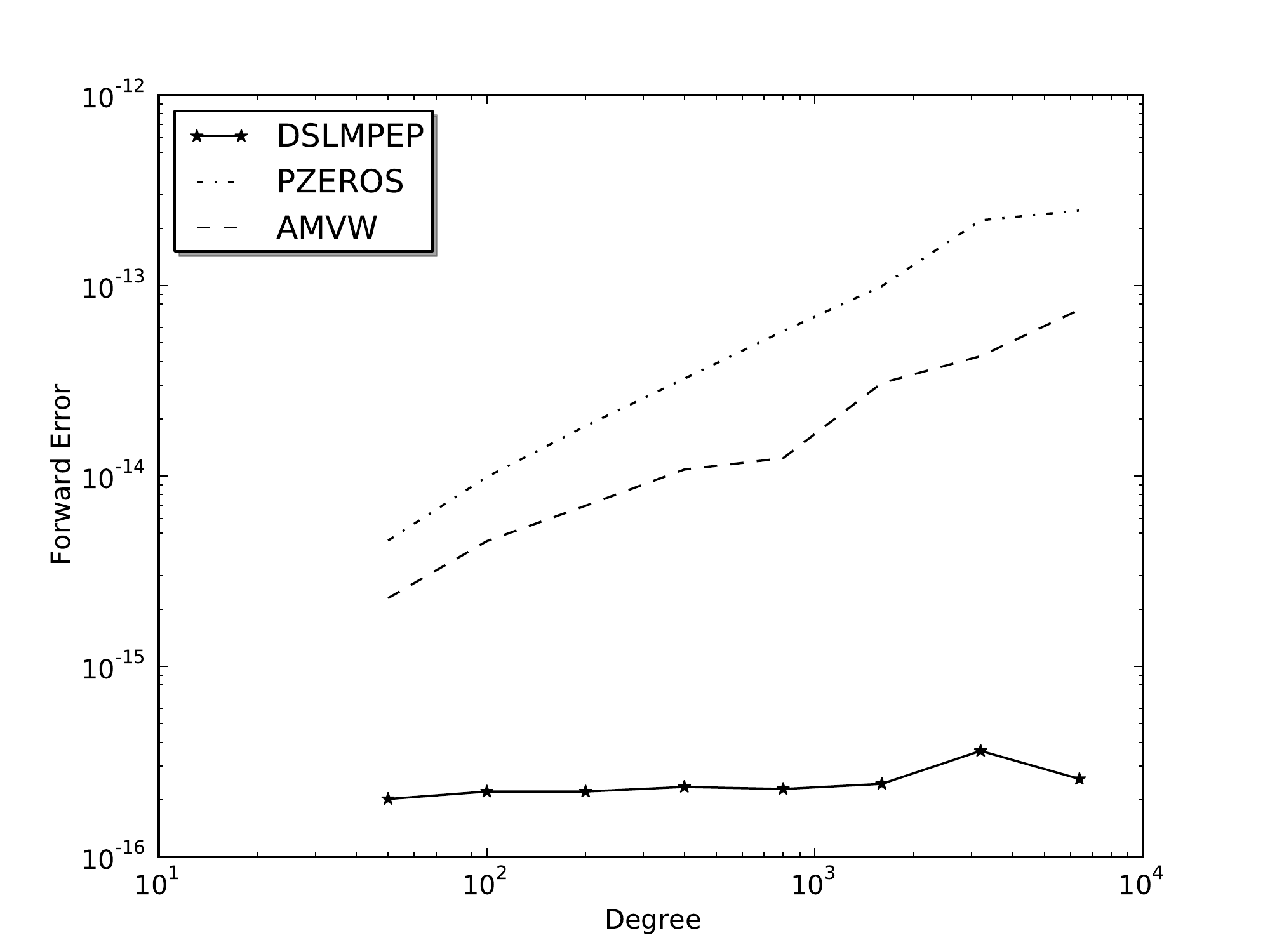}
\captionof{figure}{Test of the maximum forward error in the approximation of the roots of a scalar polynomial of degree $d$. The tests are averaged over $5$ runs and average forward errors are reported for POLZEROS and AMVW.}
\end{minipage}

\begin{minipage}{\linewidth}
\centering
\includegraphics[width=\textwidth, trim={0, 5.6cm, 0, 5.6cm}, clip]{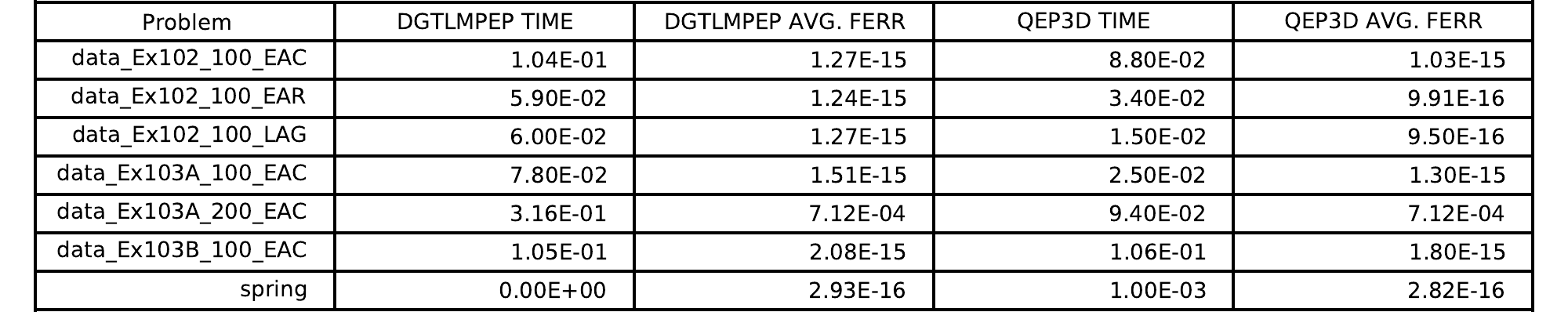}
\captionof{figure}{Average forward error and elapsed time comparisons between our method and a relevant method from the QEP3D package~\cite{Plestenjak2006}. The problems all come from the QEP3D package, with exception to the spring problem from the NLEVP package; note that our method is competitive, even though it was designed to handle more tridiagonal polynomial eigenvalue problems.}
\end{minipage}

\begin{minipage}{\linewidth}
\centering
\includegraphics[width=\textwidth, trim={0, 5.37cm, 0, 5.37cm},clip]{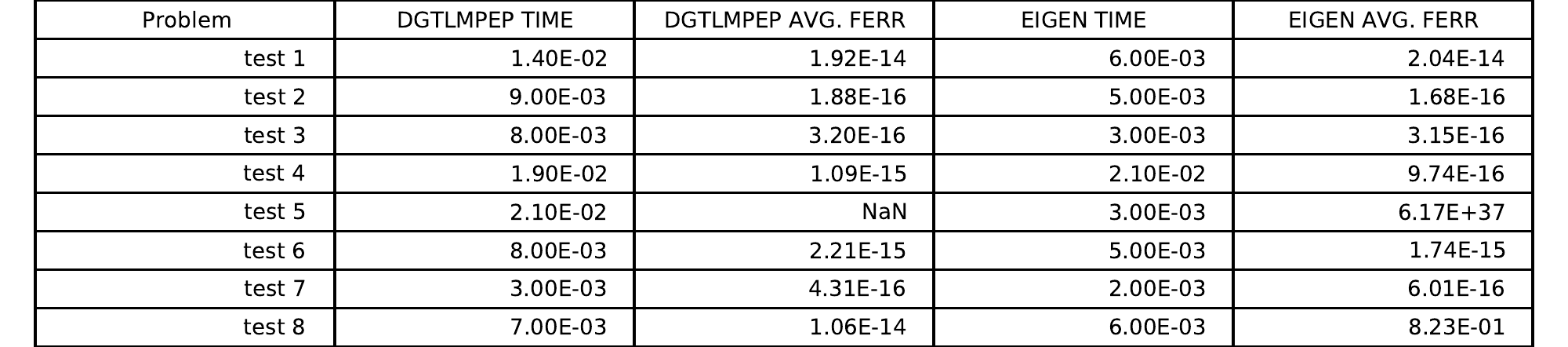}
\captionof{figure}{Average forward error and elapsed time comparisons between our method and the method EIGEN~\cite{Bini2005}. The problems also come from the EIGEN package and it is important to note that our method is competitive even though it was designed for more general tridiagonal polynomial eigenvalue problems.}
\end{minipage}

\begin{minipage}{\linewidth}
\centering
\includegraphics[width=\textwidth,trim={0, 0.65cm, 0, 0.65cm},clip]{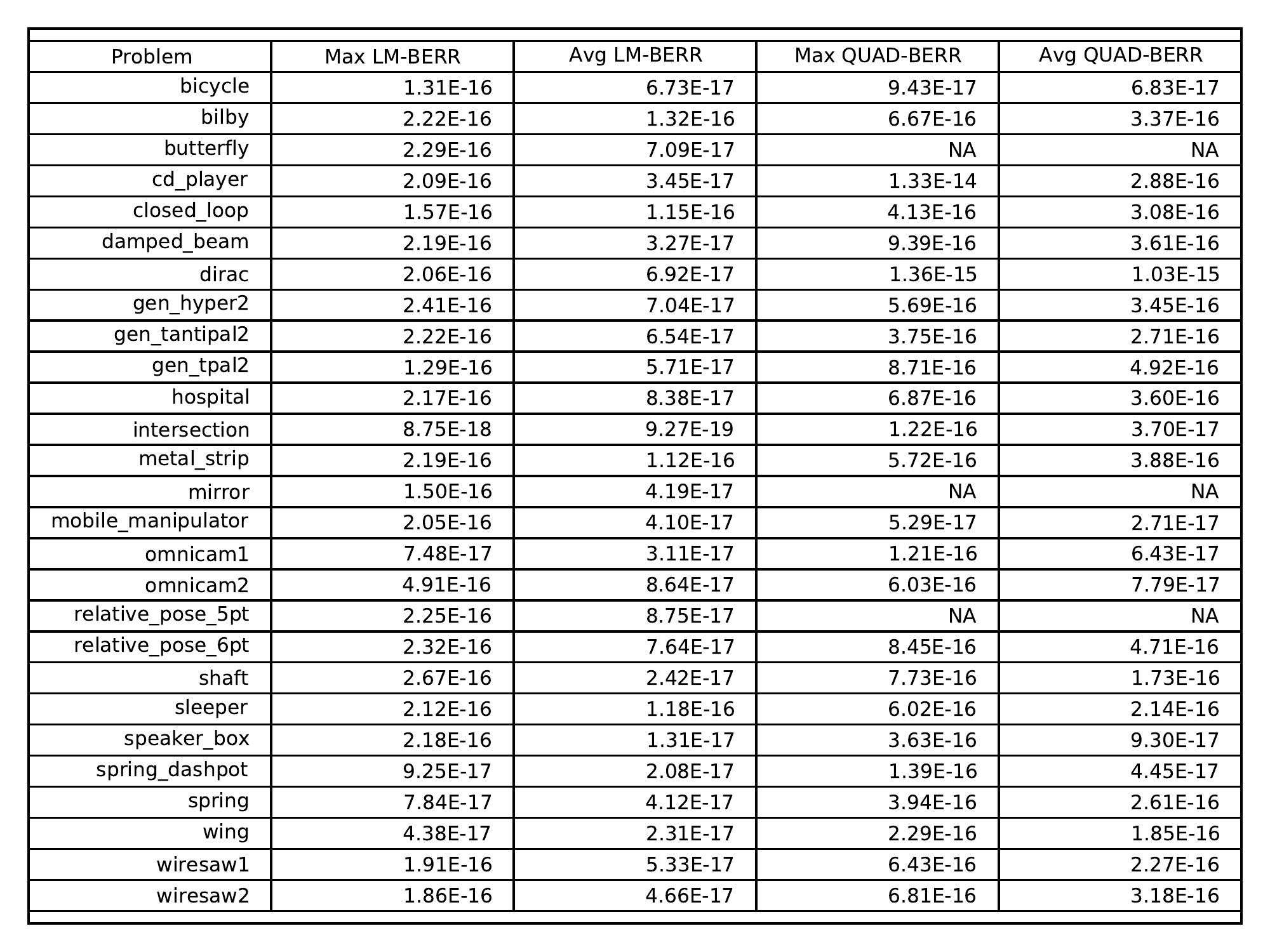}
\captionof{figure}{Comparison of the average and maximum backward error in our method and QUADEIG for many problems from the NLEVP package. The problems which QUADEIG is unable to solve are marked by NA.}
\end{minipage}

\begin{minipage}{\linewidth}
\centering
\begin{minipage}{0.5\textwidth}
\centering
\includegraphics[width=\textwidth]{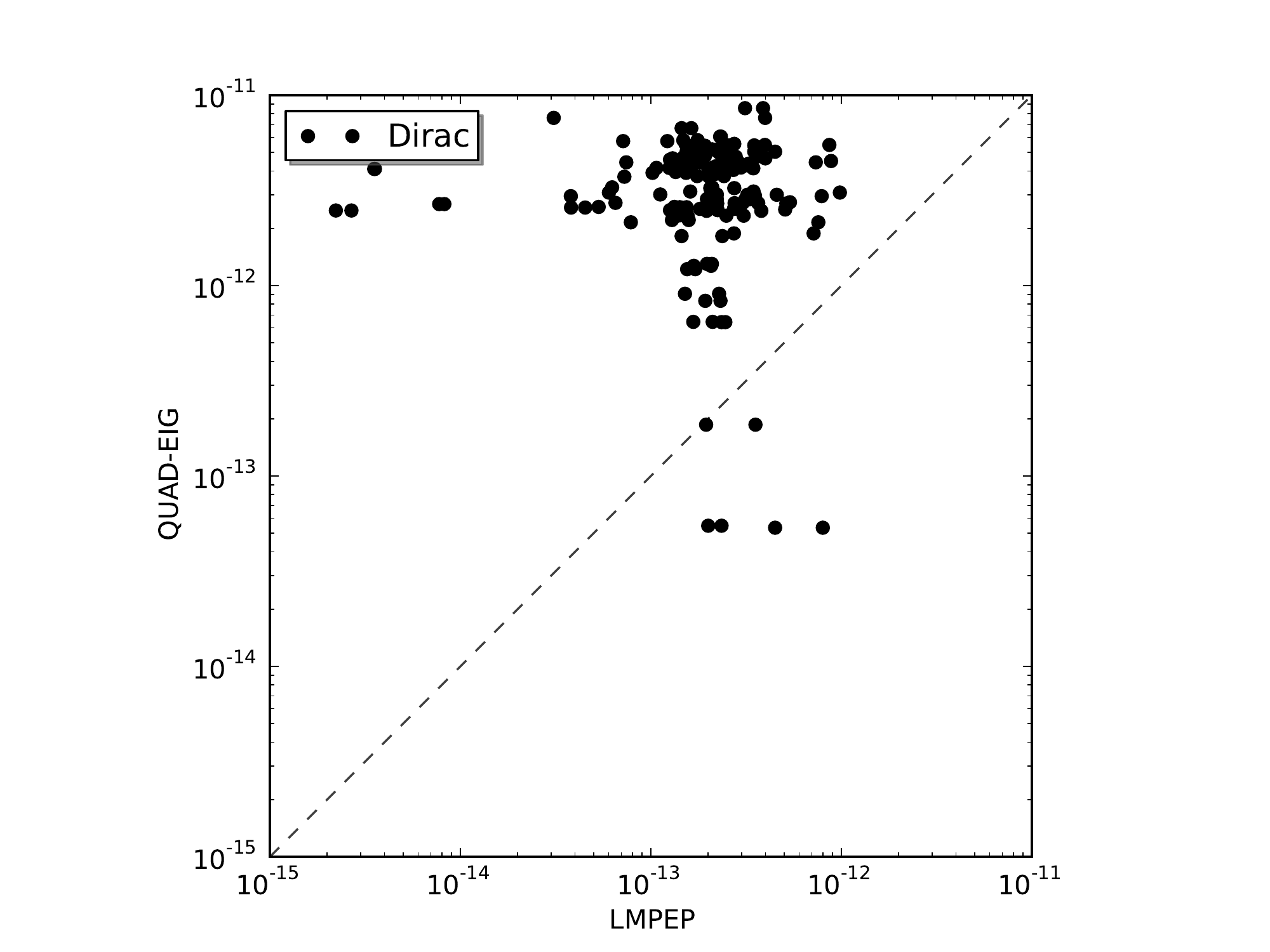}
\end{minipage}\hfill
\begin{minipage}{0.5\textwidth}
\centering
\includegraphics[width=\textwidth]{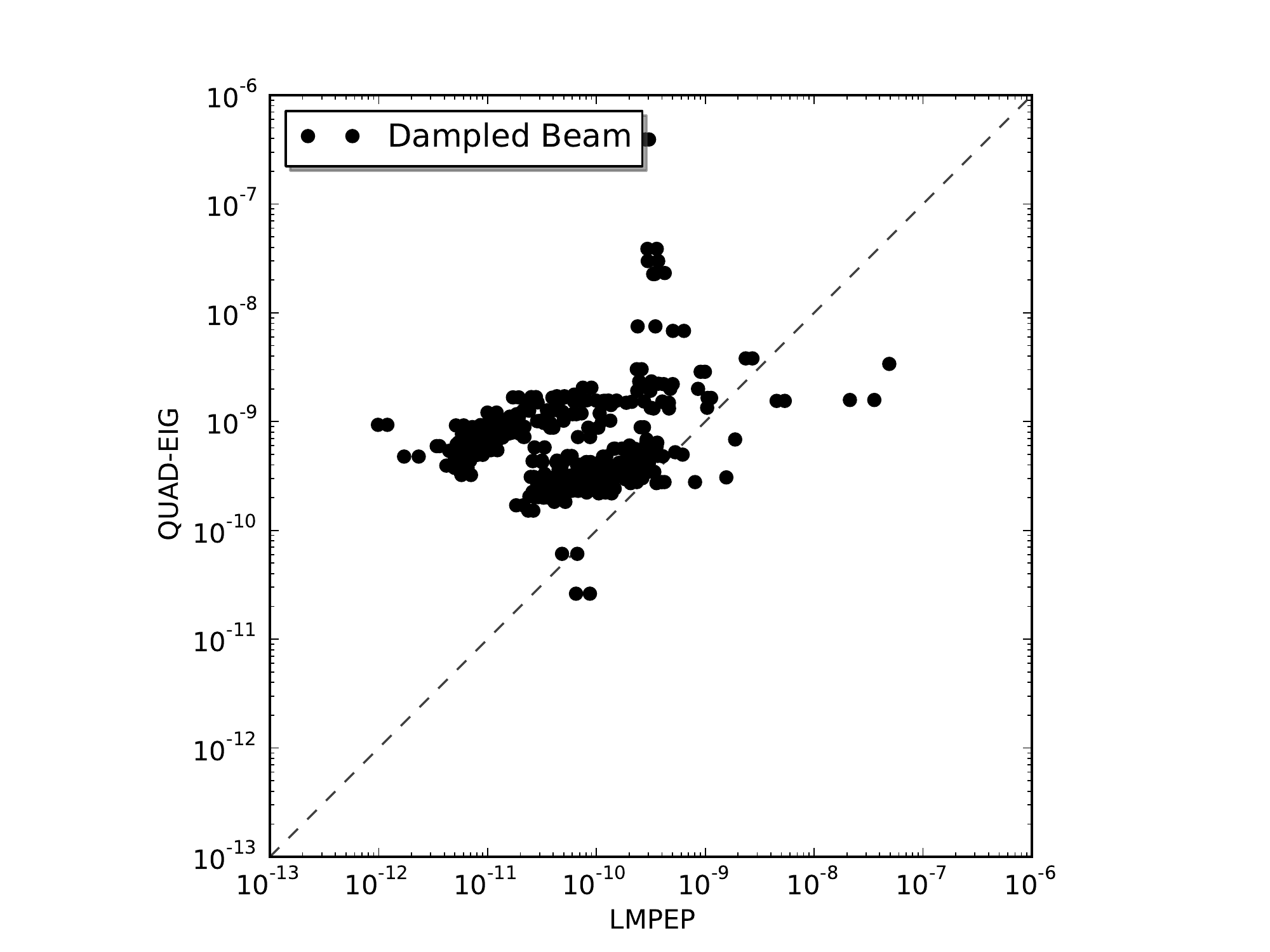}
\end{minipage}
\captionof{figure}{Comparison of the forward error in all eigenvalue approximations between our method and QUADEIG on the dirac and damped\_beam problem from the NLEVP package.}
\end{minipage}

\begin{minipage}{\linewidth}
\centering
\begin{minipage}{0.5\textwidth}
\centering
\includegraphics[width=\textwidth]{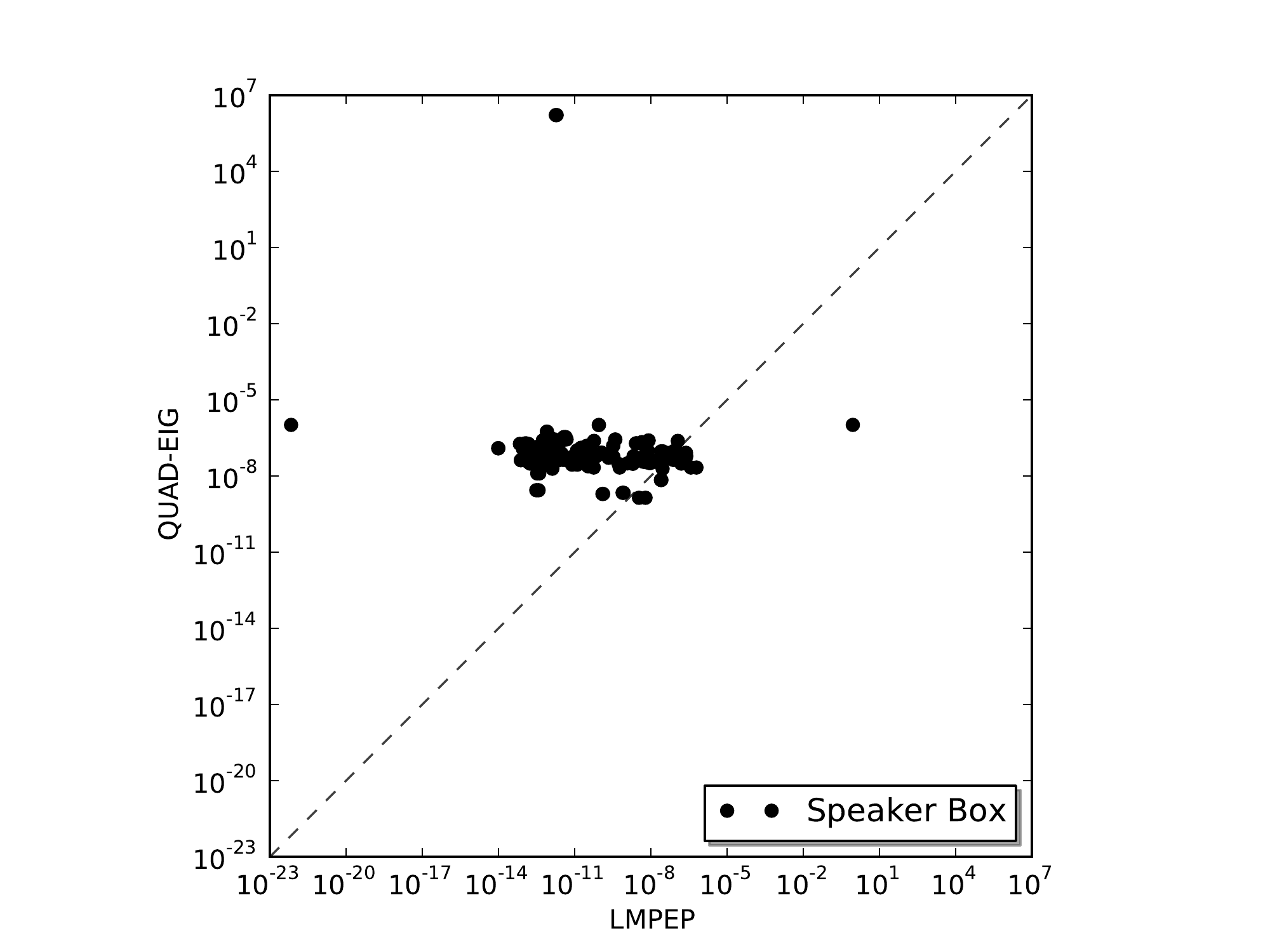}
\end{minipage}\hfill
\begin{minipage}{0.5\textwidth}
\centering
\includegraphics[width=\textwidth]{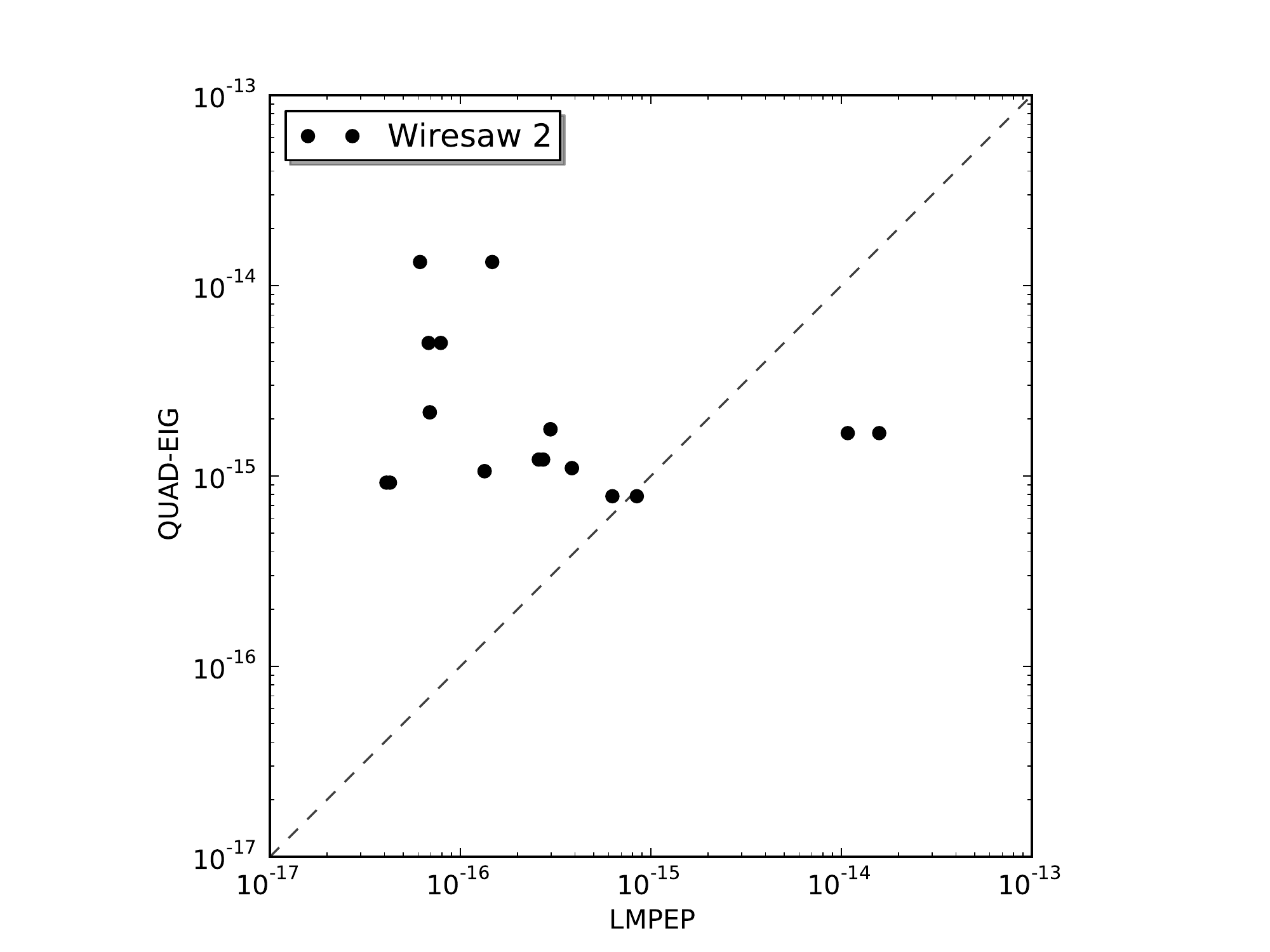}
\end{minipage}
\captionof{figure}{Comparison of the forward error in all eigenvalue approximations between our method and QUADEIG on the speaker\_box and wiresaw2 problem from the NLEVP package.}
\end{minipage}

\section{Conclusion}
Leveraging the inherent strengths of Laguerre's method and the numerical range, we have proposed a versatile, stable, and efficient method for solving the polynomial eigenvalue problem; supported by numerical experiments. Furthermore, we have demonstrated the effectiveness of our initial estimates~(\S\ref{subsec:IE}), as well as the robustness~(\S\ref{sec:rao}) and backward stability~(\S\ref{sec:bs}) of our method. 

To our knowledge, we are the first to utilize Laguerre's method with such generality as to cover such a large range of polynomial eigenvalue problems. Our method is also alone in its use of the numerical range for initial estimates. In Section~\ref{subsec:IE} we argue that these initial estimates adhere naturally to the geometry of the spectrum and we show that under suitable conditions they are no bigger in absolute value than the upper Pellet bound~(Theorem~\ref{thm:Pellet}). 

Implemented in the FORTRAN package LMPEP, numerical results attest to our method's computational complexity of $O(d^{2})$ in the scalar case, $O(d^{2}n^{2})$ in the tridiagonal case, $O(dn^{3}+d^{2}n^{3})$ in the Hessenberg case, and $O(dn^{4}+d^{2}n^{3})$ in the general case. Moreover, numerical results verify the backward stability of our method and exhibit its unprecedented level of accuracy. We are eagerly awaiting the formal release of the complete code in~\cite{Aurentz2016}, so that we can make additional comparisons to our method, especially for solving large degree polynomial eigenvalue problems. 

It would be remiss not to mention some open questions and areas worth exploration. In Theorem~\ref{thm:SAMPIE}, we show that roots of the quadratic form, under a vector of unit length, are no bigger in absolute value than the upper Pellet bound. We conjecture that they are also no smaller than the lower Pellet bound, but at this time are unable to produce a proof. We also conjecture that there are easily constructible vectors $x$ that such that the corresponding quadratic forms $x^{*}P(\lambda)x$ are distinct and their roots are within some minimal distance of the eigenvalues of $P(\lambda)$. However, we know of no such construction at the time of this writing.

In summary, we have proposed a new method for solving the polynomial eigenvalue problem that is strong in its virtues, capable of high degrees of accuracy, relatively unconstrained in its domain of operability, and promising in its possibility for future advancements.

\section{Acknowledgments}
The authors wish to acknowledge conversations with David Watkins and Dario Bini which helped construct the ideas in this paper, and we wish to thank Zdenek Strakos and an anonymous referee whose comments helped improve this paper. 


%

\end{document}